\documentclass[12pt]{amsart}
\usepackage{a4wide}
\usepackage[dvips]{epsfig}
\usepackage{amscd}
\usepackage{amssymb}
\usepackage{amsthm}
\usepackage{amsmath}
\usepackage{latexsym}
\usepackage{color}

\newtheorem{theorem}{Theorem}
\theoremstyle{plain}

\newtheorem{corollary}[theorem]{Corollary}

\newtheorem{definition}[theorem]{Definition}

\newtheorem{lemma}[theorem]{Lemma}

\newtheorem{assumptions}{Assumptions}

\newtheorem{proposition}[theorem]{Proposition}
\newtheorem{remark}[theorem]{Remark}

\numberwithin{equation}{section}
\numberwithin{theorem}{section}
\numberwithin{assumptions}{section}
\numberwithin{assumption}{section}
\newcommand{\R}{\ensuremath{\mathbb{R}}}
\newcommand{\E}{\ensuremath{\mathbb{E}}}
\newcommand{\Q}{\ensuremath{\mathbb{Q}}}
\def\e{{\mathrm{e}}}
\allowdisplaybreaks

\begin{document}
\title[Robustness of quadratic hedging strategies]{Robustness of quadratic hedging strategies in finance via backward stochastic differential equations with jumps}
\date{May 17, 2013.}
\author[Di Nunno]{Giulia Di Nunno}
\address{Center of Mathematics for Applications, University of Oslo,
PO Box 1053 Blindern, N-0316 Oslo, Norway, and, Norwegian School of Economics and Business Administration, Helleveien 30, N-5045 Bergen, Norway.}
\author[Khedher]{Asma Khedher}
\address{Chair of Mathematical Finance,
Technische Universit\"at M\"unchen,
Parkring 11, D-85748 Garching-Hochbruck, Germany,}
\author[Vanmaele]{Mich\`ele Vanmaele}
\address{ Department of Applied Mathematics, Computer Science and Statistics, Ghent University, Krijgslaan 281 S9, 9000 Gent, Belgium}
\email[]{giulian\@@math.uio.no, asma.khedher\@@tum.de, michele.vanmaele@ugent.be}
\urladdr{http://folk.uio.no/giulian/, http://users.ugent.be/~mvmaele/}




\begin{abstract}
We consider a backward stochastic differential equation with jumps (BSDEJ) which is driven by a Brownian motion and a Poisson random measure. We present two candidate-approximations to this BSDEJ and we prove that the solution of each candidate-approximation converges to the solution of the original BSDEJ in a space which we specify. We use this result to investigate in further detail the consequences of the choice of the model to (partial) hedging in incomplete markets in finance. As an application, we consider models in which the small variations in the price dynamics are modeled with a Poisson random measure with infinite activity and models in which these small variations are modeled with a Brownian motion. Using the convergence results on BSDEJs, we show that quadratic hedging strategies are robust towards the choice of the model and we derive an estimation of the model risk.
\end{abstract}

\maketitle


\section{Introduction}

Since Bismut \cite{B} introduced the theory of backward stochastic differential equations (BSDEs), there has been a wide range of literature about 
this topic. Researchers have kept on developing results on these equations and recently, many papers have studied 
BSDEs driven by L\'{e}vy processes (see, e.g., El Otmani \cite{E}, Carbone et al. \cite{CFS}, and \O ksendal and Zhang \cite{BO}). 

In this paper we consider a BSDE which is driven by a Brownian motion and a Poisson random measure (BSDEJ).  
We present two candidate-approximations to this BSDEJ and we prove that 
the solution of each candidate-approximation converges to the solution of the BSDEJ in a space which we specify.
Our aim from considering such approximations is to investigate the effect of the small jumps of the L\'{e}vy process in quadratic hedging strategies in incomplete markets in 
finance (see, e.g., F\"{o}llmer and Schweizer \cite{FSM} and Vandaele and Vanmaele \cite{VV} for more about quadratic hedging strategies in incomplete markets). 
These strategies are related to the study of the F\"{o}llmer-Schweizer decomposition (FS) or/and the Galtchouk-Kunita-Watanabe (GKW) decomposition which are both backward 
stochastic differential equations (see Choulli et al.\ \cite{CVV} for more about these decompositions). 
 
The two most popular types of quadratic hedging strategies are the locally risk-minimizing strategies and the mean-variance hedging strategies.
To explain, let us consider a market in which the risky asset is modelled by a jump-diffusion process $S(t)_{t\geq 0}$. Let $\xi$ be a contingent claim. A locally risk-minimizing strategy is a non self-financing strategy that allows a small cost process $C(t)_{t\geq 0}$ and insists on the fact that the terminal condition of the value of the portfolio  is equal to the contingent claim (see Schweizer \cite{SC}). In other words there exists a locally risk-minimizing strategy for $\xi$ if and only if $\xi$ admits a decomposition of the form 
\begin{equation}\label{FS}
\xi=\xi^{(0)}+ \int_0^T \chi^{FS}(s) dS(s) +\phi^{FS}(T),
\end{equation}
where $\chi^{FS}(t)_{t\geq 0}$ is a process such that the integral in \eqref{FS} exists and $\phi^{FS}(t)_{t\geq 0}$ is a martingale which has to satisfy certain conditions that we will show in the next sections of the paper.   
The decomposition \eqref{FS} is called the FS decomposition. Its financial importance lies in the fact that it directly provides the locally risk-minimizing strategy for $\xi$. In fact at each time $t$ the number of risky assets is given by $\chi^{FS}(t)$ and the cost $C(t)$ is given by $\phi^{FS}(t)+\xi^{(0)}$. 
The mean-variance hedging strategy is a self-financing strategy which minimizes the hedging error in mean square sense (see F\"{o}llmer and Sondermann \cite{FSO} ).

In this paper we study the robustness of these two latter hedging strategies towards the model choice.
%
 %
Hereto we assume that the process $S(t)_{t\geq 0}$ is a jump-diffusion with stochastic factors and driven by a pure jump term with infinite activity and a Brownian motion $W(t)_{t\geq 0}$.
We consider two approximations to $S(t)_{t\geq 0}$. In the first approximation $S_{0,\varepsilon}(t)_{t\geq 0}$\,, we truncate the small jumps and rescale the Brownian motion $W(t)_{t\geq 0}$ to justify the variance of the small jumps.
 In the second approximation $S_{1,\varepsilon}(t)_{t\geq 0}$\,, we truncate the small jumps and replace them by a Brownian motion $B(t)_{t\geq 0}$ independent of $W(t)_{t\geq 0}$ and scaled with the standard deviation of the small jumps. 
 
 This idea of shifting from a model with small jumps to another where those variations are modeled by some appropriately scaled continuous component goes back to Asmussen and Rosinsky \cite{AR}
  who proved that the second model approximates the first one.
 This kind of approximation results, here intended as robustness of the model, are interesting of course from the modeling point of view, but also  
 from a simulation point of view. In fact no easy algorithms are available for simulating general L\'{e}vy processes. However the approximating processes we obtain contain a compound Poisson process and a Brownian motion which are both easy to simulate (see Cont and Tankov \cite{CT}).
 
Benth et al.\ \cite{BDK, BDK2} investigated the consequences of this approximation to option pricing in finance. They consider option prices written in exponential L\'{e}vy processes and they proved the robustness of the option prices after a change of measure where the measure depends on the model choice. For this purpose the authors used Fourier transform techniques.  

 In this paper we focus mostly on the locally risk-minimizing strategies and we show that under some conditions on the parameters of the stock price process, the value of the portfolio, the amount of wealth, and the cost process in a locally risk-minimizing strategy are robust to the choice of the model. Moreover, we prove the robustness of
 the value of the portfolio and the amount of wealth in a mean-variance hedging strategy, where we assume that the parameters of the jump-diffusion are deterministic.
To prove these results we use the existence of the FS decomposition \eqref{FS} and the convergence results on BSDEJs.  
This robustness study is a continuation and a generalization of the results by Benth et al.\ \cite{BDK2}.  In fact we consider more general dynamics and we prove that indeed the locally risk-minimizing strategy  and the mean-variance hedging strategy are robust to the choice of the model. In this context we also mention a paper by Daveloose et al.\ \cite{DKV} in which  the authors studied robustness of quadratic hedging strategies using a Fourier approach and a special choice of dynamics for the price process, namely an exponential L\'{e}vy process.
The paper is organised as follows: in Section 2 we introduce the notations and we make a short introduction to BSDEJs. In Section 3 we  present the two candidate-approximations to the original BSDEJ and we prove the robustness. In Section 4 we prove the robustness of quadratic hedging strategies towards the choice of the model.
In Section 5 we conclude.
\section{Some mathematical preliminaries}
Let $(\Omega,\mathcal{F}, \mathbb{P})$ be a complete probability
space. We fix $T>0$.
Let $W=W(t)$ and $B=B(t)$, $t\in [0,T]$, be two independent standard Wiener processes and $\widetilde{N}=\widetilde{N}(dt,dz)$, $t, z \in [0,T]\times \R_0$ ($\R_0:=\R\setminus\{0\}$)
be a centered Poisson random 
measure, i.e.\ $\widetilde{N}(dt,dz)= N(dt,dz)-\ell(dz)dt$, 
where $\ell(dz)$ is the jump measure and $N(dt,dz)$ is the Poisson random measure independent of the Brownian motions $W$ and $B$ and such that $\E[N(dt,dz)]=\ell(dz)dt$. Define $\mathcal{B}(\R_0)$ as the $\sigma$-algebra
generated by the Borel sets $\bar{U} \subset \R_0$.
We assume that the jump measure has a finite second moment. Namely
$\int_{\R_0}z^2\ell(dz)<\infty.$
We introduce the $\mathbb{P}$-augmented filtrations $\mathbb{F}= (\mathcal{F}_t)_{0\leq t\leq T}$, $\mathbb{F}^\varepsilon= (\mathcal{F}^\varepsilon_t)_{0\leq t\leq T}$,
$\mathbb{G}= (\mathcal{G}_t)_{0\leq t\leq T}$, $\mathbb{G}^\varepsilon= (\mathcal{G}^\varepsilon_t)_{0\leq t\leq T}$\,, respectively by 
$$\mathcal{F}_t= \sigma \Big\{W(s), \int_0^s\int_A\widetilde{N}(du,dz), \quad s\leq t, \quad A \in \mathcal{B}(\R_0)\Big\} \vee \mathcal{N},$$
$$\mathcal{F}^\varepsilon_t= \sigma \Big\{W(s), \int_0^s\int_A\widetilde{N}(du,dz), \quad s\leq t, \quad A \in \mathcal{B}(\{|z|>\varepsilon\})\Big\} \vee \mathcal{N},$$
$$\mathcal{G}_t= \sigma \Big\{W(s), B(s), \int_0^s\int_A\widetilde{N}(du,dz), \quad s\leq t, \quad A \in \mathcal{B}(\R_0)\Big\} \vee \mathcal{N},$$
$$\mathcal{G}^\varepsilon_t= \sigma \Big\{W(s), B(s), \int_0^s\int_A\widetilde{N}(du,dz), \quad s\leq t, \quad A \in \mathcal{B}(\{|z|>\varepsilon\})\Big\} \vee \mathcal{N},$$
where $\mathcal{N}$ represents the set of $\mathbb{P}$-null events in $\mathcal{F}$\,. We introduce the notation $\mathbb{H}=(\mathcal{H}_t)_{0\leq t\leq T}$, such that $\mathcal{H}_t$ will be given
either by the $\sigma$-algebra $\mathcal{F}_t$,  $\mathcal{F}^\varepsilon_t$, $\mathcal{G}_t$ or $\mathcal{G}^\varepsilon_t$ depending on our analysis later. \\
Define the following spaces for all $\beta\geq 0$;
\begin{itemize}
\item $L^2_{T,\beta}$: the space of all $\mathcal{H}_T$-measurable random variables $X:\Omega\rightarrow \R$ such that $$\|X\|^2_{\beta}=\E[\e^{\beta T}X^2]<\infty.$$
\item $H^2_{T,\beta}$: the space of all $\mathbb{H}$-predictable processes $\phi: \Omega\times[0,T]\rightarrow \R$, such that
$$\|\phi\|_{H^2_{T,\beta}}^2=\E\Big[\int_0^T\e^{\beta t}|\phi(t)|^2dt\Big]<\infty.$$
\item $\widetilde{H}^2_{T,\beta}$: the space of all $\mathbb{H}$-adapted, c\`{a}dl\`{a}g processes $\psi:\Omega\times[0,T]\rightarrow \R$ such that
$$\|\psi\|^2_{\widetilde{H}^2_{T,\beta}}=\E\Big[\int_0^T\e^{\beta t}|\psi^2(t)dt|\Big]< \infty.$$
\item $\widehat{H}^2_{T,\beta}$: the space of all $\mathbb{H}$-predictable mappings $\theta:\Omega\times [0,T]\times \R_0\rightarrow\R$, such that 
$$\|\theta\|_{\widehat{H}^2_{T,\beta}}^2=\E\Big[\int_0^T\int_{\R_0}\e^{\beta t}|\theta(t,z)|^2\ell(dz)dt\Big]<\infty.$$
\item $S^2_{T,\beta}$: the space of all $\mathbb{H}$-adapted, c\`{a}dl\`{a}g processes $\gamma:\Omega\times[0,T]\rightarrow \R$ such that
$$\|\gamma\|^2_{S^2_{T,\beta}}=\E[\e^{\beta T}\sup_{0\leq t\leq T}|\gamma^2(t)|]< \infty.$$
\item $\nu_{\beta}=S^2_{T,\beta}\times H^2_{T,\beta}\times \widehat{H}^2_{T,\beta}$.
\item $\widetilde{\nu}_{\beta}=S^2_{T,\beta}\times H^2_{T,\beta}\times \widehat{H}^2_{T,\beta}\times H^2_{T,\beta}$.
\item $\widehat{L}_T^2(\R_0, \mathcal{B}(\R_0), \ell)$: the space of all $\mathcal{B}(\R_0)$-measurable mappings $\psi: \R_0\rightarrow \R$ such that
$$\|\psi\|^2_{\widehat{L}^2_T(\R_0, \mathcal{B}(\R_0), \ell)} = \int_{\R_0}|\psi(z)|^2\ell(dz) <\infty.$$
\end{itemize}
For notational simplicity, when $\beta=0$, we skip the $\beta$ in the notation. 


The following result is crucial in the study of the existence and uniqueness of the backward stochastic differential equations we are interested in.
Indeed it is an application of the decomposition of a random variable $\xi \in L^2_T$ with respect to orthogonal martingale random fields as integrators. See Kunita and Watanabe \cite{KW}, Cairoli and Walsh \cite{CW}, and Di Nunno and Eide \cite{DE} for the essential ideas. In Di Nunno \cite{D, D1}, and Di Nunno and Eide \cite{DE},
 explicit representations of the integrands are given in terms of the non-anticipating derivative. 
 \begin{theorem}\label{representation-theorem}
Let $\mathbb{H}=\mathbb{G}$. Every $\mathcal{G}_T$-measurable random variable $\xi \in L_T^2$ has a unique representation of the form 
\begin{align}\label{representation-theo}
\xi=\xi^{(0)}+\sum_{k=1}^3\int_0^T\int_{\R}\varphi_k(t,z) \mu_k(dt,dz),
\end{align}
where the stochastic integrators 
\begin{align*}
\mu_1(dt,dz)&=W(dt)\times \delta_0(dz), \quad \mu_2(dt,dz)=B(dt)\times \delta_0(dz), \\
\quad \mu_3(dt,dz)&=\widetilde{N}(dt,dz)\mathbf{1}_{[0,T]\times\R_0}(t,z),
\end{align*}
are orthogonal martingale random fields on $[0,T]\times\R_0$ 
and the stochastic integrands are 
$\varphi_1$, $\varphi_2 \in H^2_{T}$ and $\varphi_3 \in \widehat{H}^2_{T}$. Moreover $\xi^{(0)}=\E[\xi]$.\\
Let $\mathbb{H}=\mathbb{G}^\varepsilon$. Then for every $\mathcal{G}_T^\varepsilon$-measurable random variable $\xi \in L_T^2$, \eqref{representation-theo} holds 
with $\mu_3(dt,dz)=\widetilde{N}(dt,dz)\mathbf{1}_{[0,T]\times \{|z|>\varepsilon\}}(t,z)$.\\
Let $\mathbb{H}=\mathbb{F}$. Then for every $\mathcal{F}_T$-measurable random variable $\xi \in L_T^2$, \eqref{representation-theo} holds with $\mu_2(dt,dz)=0$.\\
Let $\mathbb{H}=\mathbb{F}^\varepsilon$. Then for every $\mathcal{F}_T^\varepsilon$-measurable random variable $\xi \in L_T^2$, \eqref{representation-theo} holds 
with $\mu_2(dt,dz)=0$ and $\mu_3(dt,dz)=\widetilde{N}(dt,dz)\mathbf{1}_{[0,T]\times \{|z|>\varepsilon\}}(t,z)$.
\end{theorem}
As we shall see the above result plays a central role in the analysis that follows. 
Let us now consider a pair $(\xi,f)$, where $\xi$ is called the terminal condition and $f$ the driver such that 

\begin{assumptions}\label{lipschitz-assumption}
\text{}\\
(A) $\xi \in L^2_T$ is $\mathcal{H}_T$-measurable\\
 (B) $f: \Omega \times [0,T]\times \R \times \R \times \R\rightarrow \R$ such that 
\begin{itemize}
\item $f(\cdot,x,y,z)$ is $\mathbb{H}$-progressively measurable for all $x,y,z$,
\item $f(\cdot,0,0,0) \in H_T^2$,
\item $f(\cdot,x,y,z)$ satisfies a uniform Lipschitz condition in $(x,y,z)$, i.e.\ there exists a constant $C$ such that for all $(x_i,y_i,z_i) \in \R\times\R\times \widehat{L}^2_T(\R_0, \mathcal{B}(\R_0),\ell)$,
$i=1,2$ we have
\begin{align*}
&|f(t,x_1,y_1,z_1)-f(t,x_2,y_2,z_2)|\\
&\qquad \leq C \Big(|x_1-x_2|+|y_1-y_2|+\|z_1-z_2\|\Big), \quad \mbox{for all } t.
\end{align*}
\end{itemize}
\end{assumptions}
We consider the following backward stochastic differential equation with jumps (in short BSDEJ)
\begin{equation}\label{bsdes}
\left\{ \begin{array}{ll} 
-dX(t) &= f(t,X(t),Y(t),Z(t,\cdot ))dt - Y(t)dW(t) - \displaystyle\int_{\R_0}Z(t,z)\widetilde{N}(dt,dz),\\
X(T) &=\xi.
 \end{array} \right.
\end{equation}
\begin{definition} 
A solution to the BSDEJ \eqref{bsdes} is a triplet of $\mathbb{H}$-adapted or predictable processes $(X,Y,Z) \in \nu$ satisfying 
\begin{align*}
X(t) &= \xi + \int_t^Tf(s,X(s),Y(s),Z(s,\cdot))ds -\int_t^T Y(s)dW(s)\\
&\qquad  - \int_t^T\int_{\R_0}Z(s,z)\widetilde{N}(ds,dz), \qquad 0\leq t\leq T.
 \end{align*}
\end{definition}
The existence and uniqueness result for the solution of the BSDEJ \eqref{bsdes} is guaranteed by the following result proved in Tang and Li \cite{TL}. 
\begin{theorem}\label{existence-and-uniqueness}
Given a pair $(\xi,f)$ satisfying Assumptions \ref{lipschitz-assumption}(A) and (B), there exists a unique solution $(X,Y,Z) \in \nu$ to the BSDEJ \eqref{bsdes}.
\end{theorem}

\section{Two candiate-approximating BSDEJs and robustness}\label{robustness}
\subsection{Two candiate-approximating BSDEJs}
In this subsection we present two candidate approximations of the BSDEJ \eqref{bsdes}. Let $\mathbb{H}=\mathbb{F}$ and $f^0$ be a function satisfying Assumptions \ref{lipschitz-assumption}(B). 
In the first candidate-approximation, we approximate the terminal condition $\xi$ of the BSDEJ \eqref{bsdes} by a sequence of random variables $\xi^0_\varepsilon \in L^2_T$, $\mathcal{F}_T$-measurable such that $$\lim_{\varepsilon\rightarrow 0}\xi^0_\varepsilon=\xi,  \quad \mbox{ in }L^2_T.$$ 
We obtain the following approximation
\begin{equation}\label{bsdes-approximation}
\left\{ \begin{array}{ll}
-d{X}_\varepsilon(t)&= f^0(t,{X}_\varepsilon(t), Y_\varepsilon(t), Z_{\varepsilon}(t,\cdot))dt -Y_\varepsilon(t)dW(t) -\displaystyle \int_{\R_0}Z_{\varepsilon}(t,z)\widetilde{N}(dt,dz),\\
X_\varepsilon(T) &= \xi^0_\varepsilon. \end{array} \right.
\end{equation}
We present the following condition on $f^0$, which we need to impose when we study the robustness results in the next section. For all $(x_i,y_i,z_i) \in \mathbb{R}\times \mathbb{R}\times \widehat{L}_T^2(\mathbb{R}_0,\mathcal{B}(\mathbb{R}_0),\ell)$, $i=1,2$, it holds that 
\begin{align}\label{condition-f-zero}
&|f(t,x_1,y_1,z_1)-f^0(t,x_2,y_2,z_2)|\nonumber\\
&\qquad \leq C \Big(|x_1-x_2|+|y_1-y_2|+\|z_1-z_2\|\Big), \quad \mbox{for all } t,
\end{align}
where $C$ is a positive constant.
In the next theorem we state the existence and uniqueness of the solution $(X_\varepsilon, Y_\varepsilon, Z_{\varepsilon}) \in \nu$ of the BSDEJ \eqref{bsdes-approximation}. This result on existence and uniqueness of the solution to \eqref{bsdes-approximation} is along the same lines as the proof of Theorem \ref{existence-and-uniqueness}, see also Tang and Li \cite{TL}. We present the proof in the Appendix, Section \ref{appendix}.

\begin{theorem}\label{existence-1}
 Let $\mathbb{H}=\mathbb{F}$. 
Given a pair $(\xi^0_\varepsilon,f^0)$ such that $\xi^0_\varepsilon \in L^2_T$ is $\mathcal{F}_T$-measurable and $f^0$ satisfies Assumptions \ref{lipschitz-assumption}(B), then there exists a unique solution $(X_\varepsilon, Y_\varepsilon, Z_{\varepsilon}) \in \nu$ to the 
BSDEJ \eqref{bsdes-approximation}.
\end{theorem}

Let $\mathbb{H}=\mathbb{G}$. We present the second-candidate approximation to \eqref{bsdes}. Hereto we introduce a sequence of random variables $\mathcal{G}_T$-measurable $\xi_\varepsilon^1 \in L^2_T$ such that $$\lim_{\varepsilon \rightarrow 0} \xi_\varepsilon^1=\xi$$ and a function $f^1$ satisfying
\begin{assumptions}\label{lipschitz-assumption-1}
$f^1: \Omega \times [0,T]\times \R \times \R \times \R\times \R\rightarrow \R$ is such that 
\begin{itemize}
\item $f^1(\cdot,x,y,z,\zeta)$ is $\mathbb{H}$-progressively measurable for all $x,y,z,\zeta$,
\item $f^1(\cdot,0,0,0,0) \in H_T^2$,
\item $f^1(\cdot,x,y,z,\zeta)$ satisfies a uniform Lipschitz condition in $(x,y,z,\zeta)$.
\end{itemize}
\end{assumptions}
Besides Assumptions \ref{lipschitz-assumption-1} which we impose on $f^1$,  we need moreover to assume the following condition in the robustness analysis later on. For all $(x_i,y_i,z_i,\zeta) \in \R\times\R\times \widehat{L}_T^2(\R_0, \mathcal{B}(\R_0), \ell)\times\R$, $i=1,2$, and for a positive constant $C$ it holds that
\begin{align}\label{stronger-Lipcshitz-condition)}
&|f(t,x_1,y_1,z_1)-f^1(t,x_2,y_2,z_2,\zeta)|\nonumber\\
&\quad \leq C \Big(|x_1-x_2|+|y_1-y_2|+\|z_1-z_2\| +|\zeta|\Big) \quad \mbox{for all } t. 
\end{align}
We introduce the second candidate BSDEJ approximation to \eqref{bsdes} which reads as follows 
\begin{equation}\label{bsdes-approximation1}
\left\{ \begin{array}{ll}
-dX_\varepsilon(t)&= f^1(t,X_\varepsilon(t),Y_\varepsilon(t),Z_{\varepsilon}(t,\cdot), \zeta_\varepsilon(t))dt -Y_\varepsilon(t)dW(t) - \displaystyle\int_{\R_0}
Z_{\varepsilon}(t,z)\widetilde{N}(dt,dz)\\
&\qquad  -\zeta_\varepsilon(t)dB(t),\\
X_\varepsilon(T) &= \xi^1_\varepsilon,
\end{array} \right.
\end{equation}
where we use the same notations as in \eqref{bsdes-approximation}. $B$ is a Brownian motion independent of $W$. Because of the presence of the additional noise $B$ the solution processes are expected to be $\mathbb{G}$-adapted (or predictable). Notice that the solution of such equation is given by $(X_\varepsilon, Y_\varepsilon, Z_{\varepsilon},\zeta_\varepsilon) \in \widetilde{\nu}$.
In the next theorem we state the existence and uniqueness of the solution of the equation \eqref{bsdes-approximation1}. The proof is very similar to the proof of Theorem \ref{existence-1}. However we work under the $\sigma$-algebra $\mathcal{G}_t$.
\begin{theorem}
 Let $\mathbb{H}=\mathbb{G}$. 
Given a pair $(\xi^1_\varepsilon,f^1)$ such that $\xi^1_\varepsilon \in L^2_T$ is $\mathcal{G}_T$-measurable and $f^1$ satisfies Assumptions \ref{lipschitz-assumption-1}, then there exists a unique solution $(X_\varepsilon, Y_\varepsilon, Z_{\varepsilon},\zeta_\varepsilon) \in \widetilde{\nu}$ to the 
BSDEJ \eqref{bsdes-approximation}.
\end{theorem}
It is expected that when \eqref{stronger-Lipcshitz-condition)} holds, the process 
$\zeta_\varepsilon$ vanishes when $\varepsilon$ goes to $0$. This will be shown in the next subsection in which we also prove the robustness of the BSDEJs. 
\subsection{Robustness of the BSDEJs}
Before we show the convergence of the two equations  \eqref{bsdes-approximation} and \eqref{bsdes-approximation1} to the BSDEJ \eqref{bsdes} when $\varepsilon$ goes to $0$, we present the following lemma in which we prove the boundedness of the solution of the equation \eqref{bsdes}. We need this lemma for our analysis in the next section.
\begin{lemma}\label{solution-boundedness}
Let $(X,Y,Z)$ be the solution of \eqref{bsdes}. Then we have for all $t\in[0,T]$, 
$$\E\Big[\int_t^TX^2(s)ds\Big]+\E\Big[\int_t^TY^2(s)ds\Big] +\E\Big[\int_t^T\int_{\R_0}Z^2(s,z)\ell(dz)ds\Big]\leq C \E[\xi^2],$$
where $C$ is a positive constant.
\end{lemma}
\begin{proof}
Recall the expression of $X$ given by \eqref{bsdes}. Applying the It\^{o} formula to $e^{\beta t} X^2(t)$ and taking the expectation, we get
\begin{align*}
\E[e^{\beta t} X^2(t)]&=\E[e^{\beta t} X^2(T)]-\beta\E\Big[\int_t^T e^{\beta s} X^2(s)ds\Big]-\E\Big[\int_t^T\e^{\beta s} Y^2(s)ds\Big]\\
&\qquad +2\E\Big[\int_t^T\e^{\beta s}X(s)f(s,X(s), Y(s), Z(s,.))ds\Big]\\
&\qquad -\E\Big[\int_t^T\int_{\R_0}\e^{\beta s}Z^2(s,z)\ell(dz)ds\Big].
\end{align*}
Thus by the Lipschitz property of $f$ we find
\begin{align*}
&\E[e^{\beta t} X^2(t)]+\E\Big[\int_t^T\e^{\beta s}Y^2(s)ds\Big] +\E\Big[\int_t^T\int_{\R_0}\e^{\beta s}Z^2(s,z)\ell(dz)ds\Big]\\
&\qquad \leq \E[e^{\beta T} X^2(T)]-\beta\E\Big[\int_t^T e^{\beta s} X^2(s)ds\Big]\\
&\qquad  \qquad +2C\E\Big[\int_t^T\e^{\beta s}X(s)\Big(|X(s)|+ |Y(s)|+ |\int_{\R_0}Z^2(s,z)\ell(dz)|^{\frac{1}{2}}\Big)ds\Big].
\end{align*}
Using the fact that for every $k>0$ and $a,b \in \R$ we have that $2ab\leq ka^2+\frac{b^2}{k}$ and $(a+b+c)^2\leq 3(a^2+b^2+c^2)$, choosing $\beta=6C^2+1$, and noticing that $\beta>0$, the result follows.
\end{proof}
From now on we use a unified notation for both BSDEJs \eqref{bsdes-approximation} and \eqref{bsdes-approximation1} in the BSDEJ
\begin{equation}\label{both-approximations}
\left\{ \begin{array}{ll}
-dX^\rho_\varepsilon(t)&= f_\varepsilon^\rho(t)dt -Y^\rho_\varepsilon(t)dW(t) - \displaystyle\int_{\R_0}
Z^\rho_{\varepsilon}(t,z)\widetilde{N}(dt,dz)-\zeta^\rho_\varepsilon(t)dB(t),\\
X^\rho_\varepsilon(T) &= \xi_\varepsilon^\rho, \qquad \mbox{for } \rho=0\mbox{ and } \rho=1,
\end{array} \right.
\end{equation}
where 
$$
f_\varepsilon^\rho(t) = \left\{
    \begin{array}{ll}
        f^0(t,X_\varepsilon^0(t), Y_\varepsilon^0(t), Z_\varepsilon^0(t)), &\quad \rho=0, \\
        f^1(t,X_\varepsilon^1(t), Y_\varepsilon^1(t), Z_\varepsilon^1(t),\zeta_\varepsilon^1(t)), &\quad  \rho=1 
    \end{array}
\right.
$$
and 
$$
\zeta_\varepsilon^\rho(t) = \left\{
    \begin{array}{ll}
       0, &\quad \rho=0, \\
       \zeta_\varepsilon^1(t), &\quad  \rho=1.
    \end{array}
\right.
$$
Notice that the BSDEJ \eqref{both-approximations} has the same solution as \eqref{bsdes-approximation} and \eqref{bsdes-approximation1} respectively for $\rho=0$ and $\rho=1$. 
We state the following theorem in which we prove the convergence of both BSDEJs \eqref{bsdes-approximation} and \eqref{bsdes-approximation1} to the BSDEJ \eqref{bsdes}.
\begin{theorem}\label{bsdej-robustness}
 Assume that $f^0$ and $f^1$ satisfy \eqref{condition-f-zero} and \eqref{stronger-Lipcshitz-condition)} respectively.
Let $(X,Y,Z)$ be the solution of  \eqref{bsdes} and $(X^\rho_\varepsilon, Y^\rho_\varepsilon, Z^\rho_{\varepsilon},\zeta^\rho_\varepsilon)$ be the solution of  \eqref{both-approximations}.
Then we have for $t\in [0,T]$,
\begin{align*}
&\E\Big[\int_t^T|X(s)-X^\rho_\varepsilon(s)|^2ds\Big]+ \E\Big[\int_t^T|Y(s)-Y^\rho_\varepsilon(s)|^2ds\Big]\nonumber\\
&\qquad +\E\Big[\int_t^T\int_{\R_0}|Z(s,z)-Z^\rho_{\varepsilon}(s,z)|^2\ell(dz)ds\Big]+ \E\Big[\int_t^T |\zeta^\rho_\varepsilon(s)|^2ds\Big] \nonumber\\
&\qquad \qquad \leq K\E[|\xi-\xi^\rho_\varepsilon|^2], \qquad \mbox{for } \rho=0 \mbox{ and } \rho=1,
\end{align*}
where $K$ is a positive constant. 
\end{theorem}
\begin{proof}
Let 
\begin{align}\label{bar}
\bar{X}^\rho_\varepsilon(t)&=X^\rho(t)-X^\rho_\varepsilon(t), \quad \bar{Y}^\rho_\varepsilon(t)=Y^\rho(t)-Y^\rho_\varepsilon(t), \quad \bar{Z}^\rho_\varepsilon(t,z)=Z^\rho(t,z)-Z^\rho_{\varepsilon}(t,z),   \nonumber\\
\bar{f}^\rho_\varepsilon(t)&= f(t,X(t), Y(t), Z(t,.))-f_\varepsilon^\rho(t).
\end{align}
Applying the It\^{o} formula to $\e^{\beta t}|\bar{X}^\rho_\varepsilon(t)|^2$, we get
\begin{align}\label{ito-formula}
&\E[\e^{\beta t}|\bar{X}^\rho_\varepsilon(t)|^2]+ \E\Big[\int_t^T\e^{\beta s}|\bar{Y}^\rho_\varepsilon(s)|^2ds \Big]+\E\Big[\int_t^T\int_{\R_0} \e^{\beta s}|\bar{Z}^\rho_{\varepsilon}(s,z)|^2\ell(dz)ds\Big]
\nonumber\\
&\qquad + \E\Big[\int_t^T \e^{\beta s}|\zeta^\rho_\varepsilon(s)|^2ds\Big]\nonumber\\
&\qquad \qquad =\E[e^{\beta T}|\bar{X}^\rho_\varepsilon(T)|^2] -\beta\E \Big[\int_t^T\e^{\beta s}|\bar{X}^\rho_\varepsilon(s)|^2ds \Big]  + 2\E \Big[\int_t^T\e^{\beta s}\bar{X}^\rho_\varepsilon(s)\bar{f}^\rho_\varepsilon(s)ds \Big].
\end{align}
Using conditions \eqref{condition-f-zero} and \eqref{stronger-Lipcshitz-condition)}, we get
\begin{align*}
&\E[\e^{\beta t}|\bar{X}^\rho_\varepsilon(t)|^2]+ \E[\int_t^T\e^{\beta s}|\bar{Y}^\rho_\varepsilon(s)|^2ds] +\E\Big[\int_t^T \int_{\R_0}\e^{\beta s}|\bar{Z}^\rho_{\varepsilon}(s,z)|^2\ell(dz)ds\Big] \\
&\qquad \qquad+\E\Big[\int_t^T\int_{\R_0}\e^{\beta s}|\zeta^\rho_\varepsilon(s)|^2ds\Big]\\
&\qquad   \leq \E[e^{\beta T}|\bar{X}^\rho_\varepsilon(T)|^2] -\beta\E\Big[\int_t^T\e^{\beta s}|\bar{X}^\rho_\varepsilon(s)|^2ds\Big]  \\
&\qquad \qquad + 2C \E\Big[\int_t^T \e^{\beta s}|\bar{X}^\rho_\varepsilon(s)|\Big(|\bar{X}^\rho_\varepsilon(s)|+|\bar{Y}^\rho_\varepsilon(s)|+|\zeta^\rho_\varepsilon(s)|+(\int_{\R_0}|\bar{Z}^\rho_{\varepsilon}(s,z)|^2\ell(dz))^{\frac{1}{2}}\Big)ds\Big].
\end{align*}
Using the fact that for every $k>0$ and $a,b \in \R$ we have that $2ab\leq ka^2+\frac{b^2}{k}$ and $(a+b+c+d)^2\leq 4(a^2+b^2+c^2+d^2)$, we obtain
\begin{align*}
&\E[\e^{\beta t}|\bar{X}^\rho_\varepsilon(t)|^2]+ \E[\int_t^T\e^{\beta s}|\bar{Y}^\rho_\varepsilon(s)|^2ds]+\E\Big[\int_t^T\int_{\R_0}\e^{\beta s}|\bar{Z}^\rho_{\varepsilon}(s,z)|^2\ell(dz)ds\Big] \\
&\qquad \qquad  +\E\Big[\int_t^T\e^{\beta s}|\zeta^\rho_\varepsilon(s)|^2ds\Big]\\
&\qquad \leq \E[e^{\beta T}|\bar{X}^\rho_\varepsilon(T)|^2] -\beta\E\Big[\int_t^T\e^{\beta s}|\bar{X}^\rho_\varepsilon(s)|^2ds\Big]  +
 8C^2 \E\Big[\int_t^T \e^{\beta s}|\bar{X}^\rho_\varepsilon(s)|^2ds\Big]\\
&\qquad \qquad + \frac{1}{2}
\E\Big[\int_t^T\e^{\beta s}|\bar{X}^\rho_\varepsilon(s)|^2ds\Big]+ \frac{1}{2}\E\Big[\int_t^T\e^{\beta s}|\zeta^\rho_\varepsilon(s)|^2ds\Big]\\
&\qquad \qquad +\frac{1}{2} \E\Big[\int_t^T\e^{\beta s}|\bar{Y}^\rho_\varepsilon(s)|^2ds\Big]
+\frac{1}{2}\E\Big[\int_t^T\int_{\R_0}\e^{\beta s}|\bar{Z}^\rho_\varepsilon(s,z)|^2\ell(dz)ds\Big].
\end{align*}
Choosing $\beta=8C^2+1$ and since $\E[\e^{\beta t}|\bar{X}^\rho_\varepsilon(t)|^2]>0$, we get
\begin{align*}
&\E\Big[\int_t^T\e^{\beta s}|\bar{X}^\rho_\varepsilon(s)|^2ds\Big]+ \E\Big[\int_t^T\e^{\beta s}|\bar{Y}^\rho_\varepsilon(s)|^2ds\Big]+\E\Big[\int_t^T\int_{\R_0}\e^{\beta s}|\bar{Z}^\rho_\varepsilon(s,z)|^2\ell(dz)ds\Big] \\
&\qquad \qquad +\E\Big[\int_t^T\e^{\beta s}|\zeta^\rho_\varepsilon(s)|^2ds\Big]\\
&\qquad \leq K\E[\e^{\beta T}|\bar{X}^\rho_\varepsilon(T)|^2],
\end{align*}
where $K$ is a positive constant and the result follows using the fact that $\beta >0$.  
\end{proof}
\begin{remark}
Notice that since $\mathcal{F}_t\subset \mathcal{G}_t$, the solution of \eqref{bsdes} is also $\mathcal{G}_t$ adapted. This fact allowed us to compare the solution of \eqref{bsdes} with the solution of \eqref{bsdes-approximation1}.
\end{remark}
In the last theorem, we proved the convergence of the solution of \eqref{bsdes-approximation} respectively \eqref{bsdes-approximation1} to the solution of \eqref{bsdes} in the space $\widetilde{H}^2_{T} \times H^2_{T}\times \widehat{H}^2_{T,}$ respectively $\widetilde{H}^2_{T} \times H^2_{T}\times \widehat{H}^2_{T}\times H^2_{T}$. In the next proposition we prove the convergence in $\nu$, respectively $\widetilde{\nu}$.
\begin{proposition}\label{robustness-X}
Assume that \eqref{condition-f-zero} and \eqref{stronger-Lipcshitz-condition)} hold. Let $X$, $X^\rho_\varepsilon$ be the solution of \eqref{bsdes}, \eqref{both-approximations}, respectively. Then we have
$$\E\Big[\sup_{0\leq t\leq T}|X(t) - X^\rho_\varepsilon(t)|^2\Big] \leq C\E[|\xi -\xi^\rho_\varepsilon|^2], \qquad \mbox{for } \rho=0 \mbox{ and } \rho=1,$$
where $C$ is a positive constant.
\end{proposition}
\begin{proof}
Let $\bar{X^\rho_\varepsilon}$, $\bar{Y^\rho_\varepsilon}$, $\bar{Z^\rho_\varepsilon}$, and $\bar{f}^\rho_\varepsilon$ be as in \eqref{bar}. Then applying H\"{o}lder's inequality, we have for $K>0$
\begin{align*}
\E\Big[\sup_{0\leq t\leq T} |\bar{X}^\rho_\varepsilon(t)|^2\Big]&\leq K \Big(\E\Big[|\bar{X}_\varepsilon^\rho(T)|^2\Big]+\E\Big[\int_0^T|\bar{f}^\rho_\varepsilon(s)|^2ds\Big] + \E\Big[\sup_{0\leq t\leq T}|\int_t^T\bar{Y}^\rho_\varepsilon(s)dW(s)|^2\Big]\\
&\qquad+\E\Big[\sup_{0\leq t\leq T}|\int_t^T\int_{\R_0}\bar{Z}^\rho_{\varepsilon}(s,z)\widetilde{N}(ds,dz)|^2\Big]\\
&\qquad +\E\Big[\sup_{0\leq t\leq T}|\int_t^T\zeta^\rho_\varepsilon(s)dB(s)|^2\Big]\Big).
\end{align*}
However from Burkholder's inequality we can prove that for $C>0$, we have (for more details see Tang and Li \cite{TL}) 
 \begin{align*}
 \E\Big[\sup_{0\leq t\leq T}|\int_t^T\int_{\R_0}\bar{Z}^\rho_\varepsilon(s,z)\widetilde{N}(ds,dz)|^2\Big]&\leq
 C\E\Big[\int_0^T\int_{\R_0}|\bar{Z}^\rho_{\varepsilon}(s,z)|^2\ell(dz)ds\Big],\\
 \E\Big[\sup_{0\leq t\leq T}|\int_t^T\bar{Y}^\rho_\varepsilon(s)dW(s)|^2\Big]&\leq
 C\E\Big[\int_0^T|\bar{Y}^\rho_\varepsilon(s)|^2ds\Big],\\
  \E\Big[\sup_{0\leq t\leq T}|\int_t^T\zeta^\rho_\varepsilon(s) dB(s)|^2\Big]&\leq
 C\E\Big[\int_0^T|\zeta^\rho_\varepsilon(s)|^2ds\Big].
 \end{align*}
 Thus from the estimates on $f^0$ and $f^1$ in equations \eqref{condition-f-zero} and \eqref{stronger-Lipcshitz-condition)} and Theorem \ref{bsdej-robustness} we get the result. 
\end{proof}
Notice that we proved the convergence of the two candidate approximating BSDEJs \eqref{bsdes-approximation}, \eqref{bsdes-approximation1} to 
the BSDEJ \eqref{bsdes} in the space $\nu$, $\widetilde{\nu}$ respectively. This type of convergence is stronger 
than the $L^2$-convergence.

\section{Robustness of the F\"ollmer-Schweizer decomposition with applications to partial-hedging in finance}\label{finance}
We assume we have two assets. One of them is a riskless asset with price $S^{(0)}$ given by
$$dS^{(0)}(t) = S^{(0)}(t)r(t)dt,$$
where $r(t)=r(t,\omega) \in \R$ is the short rate. The dynamics of the risky asset are given by
\begin{align*}
\left\{ \begin{array}{ll}
dS^{(1)}(t) &= S^{(1)}(t)\Big\{a(t)dt +b(t)dW(t)+\displaystyle\int_{\R_0}\gamma(t,z)\widetilde{N}(dt,dz)\Big\},\\
S^{(1)}(0)&=x \in \R_+\,,
\end{array}\right.
\end{align*}
where $a(t)=a(t,\omega)\in \R$, $b(t)=b(t,\omega)\in \R$, and $\gamma(t,z)=\gamma(t,z,\omega) \in \R$ for $t\geq 0$, $z\in \R_0$ are adapted processes. We assume that $\gamma(t,z,\omega)=g(z)\widetilde{\gamma}(t,\omega)$, such that 
\begin{equation}\label{variance-term}
G^2(\varepsilon):=\int_{|z|\leq\varepsilon}g^2(z)\ell(dz)<\infty.
\end{equation}
The dynamics of the 
 discounted price process $\widetilde{S}=\frac{S^{(1)}}{S^{(0)}}$ are given by 
 \begin{align}\label{price-process}
d\widetilde{S}(t) &= \widetilde{S}(t)\Big[(a(t)-r(t)) dt + b(t) dW(t)+\int_{\R_0}\gamma(t,z)\widetilde{N}(dt,dz)\Big].
\end{align}
Since $\widetilde{S}$ is a semimartingale, we can decompose it into a local martingale $M$ starting at zero in zero and a finite variation process $A$, with $A(0)=0$, where $M$ and $A$ have the following expressions 
\begin{align}\label{local-martingale}
M(t)&=\int_0^t b(s)\widetilde{S}(s)dW(s)+\int_0^t \int_{\R_0}\gamma(s,z)\widetilde{S}(s)\widetilde{N}(ds,dz),
\end{align}
\begin{align*}
A(t)&=\int_0^t(a(s)-r(s))\widetilde{S}(s) ds.
\end{align*}
We denote the predictable compensator associated to $M$ (see Protter \cite{P}) by
$$\langle M\rangle(t) = \int_0^tb^2(s) \widetilde{S}^2(s)ds +  \int_0^t \int_{\R_0}\widetilde{S}^2(s)\gamma^2(s,z)\ell(dz)ds$$
and we can represent the process $A$ as follows 
\begin{align}\label{finite-variation1}
A(t)=\int_0^t\frac{a(s)-r(s)}{ \widetilde{S}(s)\big(b^2(s) + \int_{\R_0}\gamma^2(s,z)\ell(dz)\big)}d\langle M\rangle(s).
\end{align}
Let $\alpha$ be the integrand in equation \eqref{finite-variation1}, that is the process given by
\begin{align}\label{alpha-process}
\alpha(t) :=\frac{a(t)-r(t)}{\widetilde{S}(t)\big( b^2(t) +  \int_{\R_0}\gamma^2(t,z)\ell(dz)\big)}, \qquad 0\leq t\leq T.
\end{align}
We define a process $K$ by means of $\alpha$ as follows 
\begin{align}\label{mvt}
K(t) = \int_0^t \alpha^2(s)d\langle M\rangle(s) = \int_0^t \frac{(a(s)-r(s))^2}{b^2(s)  + \int_{\R_0}\gamma^2(s,z)\ell(dz)}ds.
\end{align}
The process $K$ is called the mean-variance-trade-off (MVT) process. 

Since the stock price fluctuations are modeled by jump-diffusion then the market is incomplete and not every contingent claim can be replicated by a self-financing strategy and there is no perfect hedge. However, one can adopt a partial hedging strategy according to some optimality criteria minimizing the risk.
F\"{o}llmer and Schweizer \cite{FSM} introduced the so-called quadratic hedging strategies. 
The study of such strategies heavily depends on the F\"{o}llmer-Schweizer (FS) decomposition.  
This decomposition was first introduced by F\"{o}llmer and Schweizer \cite{FSM} for the continuous case and extended to the discontinuous case 
by Ansel and Stricker \cite{AS}.

In order to formulate our robustness study for the quadratic hedging strategies, we present the definition of the FS decomposition.
We first introduce the following notations. Let $S$ be a semimartingale. Then $S$ can be decomposed as follows $S=S(0)+ M +A,$
where $S(0)$ is finite-valued and $\mathcal{F}_0$-measurable, $M$ is a local martingale with $M(0)=0$, and $A$ is a finite variation process with $A(0)=0$. 
We denote by $L(S)$, the $S$-integrable processes, that is the class of predictable processes for which we can determine the stochastic integral with respect to $S$.  
We define the space $\Theta$ by
$$\Theta:= \Big\{ \theta \in L(S)\, |\,  \E\Big[\int_0^T\theta^2(s)d\langle M\rangle(s) + \big(\int_0^T|\theta(s)dA(s)|\big)^2\Big] <\infty\Big\}.$$
Now we give the definition of the FS decomposition.
\begin{definition}
Let $S$ be a semimartingale. An $\mathcal{F}_T$-measurable and square integrable random variable $H$ admits a F\"{o}llmer-Schweizer decomposition if there exist a constant $H_0$, an $S$-integrable process $\chi^{FS} \in \Theta$, and a square integrable martingale $\phi^{FS}$ such that $\phi^{FS}$ is orthogonal to $M$ and  
\begin{align*}
H=H_0+\int_0^T \chi^{FS}(s)dS(s) + \phi^{FS}(T).
\end{align*}
\end{definition}
Monat and Stricker \cite{MS} show that a sufficient condition for the existence of the FS decomposition is to assume that the MVT process $K$ given by \eqref{mvt} is uniformly bounded. The most general result concerning the existence and uniqueness of the FS decomposition is given by Choulli et al.~\cite{C}.
In our case we assume that the process $K$ is uniformly bounded in $t$ by a constant $C$. Thus for all $t \in [0,T]$, we have 
\begin{align}\label{condition-mvt}
K(t)= \int_0^t \frac{(a(s)-r(s))^2}{b^2(s)  + \int_{\R_0}\gamma^2(s,z)\ell(dz)}ds<C, \qquad \mathbb{P}\mbox{-a.s.}
\end{align}
Under the latter condition we can define the minimal martingale density by 
\begin{equation}\label{signed}
\mathcal{E}\Big(\int_0^. \alpha(s)dM(s)\Big)_t, 
\end{equation}
where 
$M$ and $\alpha$ are respectively given by \eqref{local-martingale} and \eqref{alpha-process} and 
 $\mathcal{E}(X)$ is the exponential martingale for $X$ (see Theorem II, 37 in Protter \cite{P} for a general formula for exponential martingales).
Notice that \eqref{signed} defines a signed minimal martingale measure. For this measure to be a probability measure we have to assume that $\mathcal{E}\Big(\int_0^. \alpha(s)dM(s)\Big)_t>0$, $\forall t\in [0,T]$ (see, e.g., Choulli et al.\ \cite{CVV}).
This latter condition is equivalent to 
\begin{align}\label{existence-mmm}
\widetilde{S}(t)\alpha(t)\gamma(t,z)>-1, \qquad \mbox { a.e.\ in }  (t,z,\omega)
\end{align}
(see Proposition 3.1 in Arai \cite{A}).
In the following we assume that \eqref{existence-mmm} holds.
Let $\xi$ be a square integrable contingent claim and $\widetilde{\xi}=\frac{\xi}{S^{(0)}(T)}$ its discounted value.
Let $ {\frac{d\widetilde{\mathbb{Q}}}{d\mathbb{P}}|}_{\mathcal{F}_t} :=\mathcal{E}\Big(\int_0^. \alpha(s)dM(s)\Big)_t$ be the minimal martingale
measure. Define 
$\widetilde{V}(t)=\E_{\widetilde{\mathbb{Q}}}[\widetilde{\xi}|\mathcal{F}_t]$. Then from Proposition 4.2 in Choulli et al.\ \cite{CVV},  we have the following
FS decomposition for $\widetilde{V}$ written under the world measure $\mathbb{P}$ 
\begin{align}\label{portfolio-general}
\widetilde{V}(t)=\E_{\widetilde{\mathbb{Q}}}[\widetilde{\xi}]+\int_0^t \chi^{FS}(s) d\widetilde{S}(s)+ \phi^{FS}(t),
\end{align}
where $\phi^{FS}$ is a $\mathbb{P}$-martingale orthogonal to $M$ and $\chi^{FS} \in \Theta$.   
Replacing $\widetilde{S}$ by its value \eqref{price-process} in \eqref{portfolio-general} we get
\begin{equation}\label{bsde-portfolio}
\left\{ \begin{array}{ll}
d\widetilde{V}(t)&=\widetilde{\pi}(t)(a(t)-r(t))dt +\widetilde{\pi}(t)b(t)dW(t)\\
& \qquad \qquad+ \displaystyle\int_{\R_0}\widetilde{\pi}(t)\gamma(t,z)\widetilde{N}(dt,dz)+ d\phi^{FS}(t),\\
\widetilde{V}(T)&=\widetilde{\xi},
\end{array} \right.
\end{equation}
where $\widetilde{\pi}= \chi^{FS}\widetilde{S}$. 

The financial importance of such decomposition lies in the fact that it directly provides the locally risk-minimizing strategy. In fact the component $\widetilde{\pi}(t)$ is the amount of wealth $V(t)$ to invest in the stock at time $t$ and $\phi^{FS}(t)+\E_{\widetilde{\mathbb{Q}}}[\widetilde{\xi}]$ is the cost process in a risk-minimizing strategy.  

Since $\phi^{FS}(T)$ is a $\mathcal{F}_T$-measurable square integrable martingale then applying Theorem \ref{representation-theo} with $\mathbb{H}=\mathbb{F}$ and the martingale property of $\phi^{FS}(T)$ we know that  there exist stochastic integrands $Y^{FS}$, $Z^{FS}$, such that 
\begin{align*}
\phi^{FS}(t)=\E[\phi^{FS}(T)]+ \int_0^t Y^{FS}(s)dW(s) + \int_0^t\int_{\R_0}Z^{FS}(s,z) \widetilde{N}(ds,dz).
\end{align*}
Since $\phi^{FS}$ is a martingale then we have $\E[\phi^{FS}(T)]=\E[\phi^{FS}(0)]$.
However from \eqref{portfolio-general} we deduce that $\phi^{FS}(0)=0$. Therefore
\begin{align}\label{phi}
\phi^{FS}(t)= \int_0^t Y^{FS}(s)dW(s) + \int_0^t\int_{\R_0}Z^{FS}(s,z) \widetilde{N}(ds,dz).
\end{align}
In view of the orthogonality of $\phi^{FS}$ and $M$, we get
\begin{align}\label{orthogonality}
Y^{FS}(t)b(t) +\int_{\R_0}Z^{FS}(t,z)\gamma(t,z)\ell(dz) = 0.
\end{align}
In that case, the set of equations \eqref{bsde-portfolio} are equivalent to
\begin{equation}\label{portfolio}
\left\{ \begin{array}{ll}
d\widetilde{V}(t)&=\widetilde{\pi}(t)(a(t)-r(t))dt +\big(\widetilde{\pi}(t) b(t)+Y^{FS}(t)\big)dW(t)\\
&\qquad \qquad  + \displaystyle\int_{\R_0}\big(\widetilde{\pi}(t)\gamma(t,z)+Z^{FS}(t,z)\big)\widetilde{N}(dt,dz),\\
\widetilde{V}(T)&=\widetilde{\xi}.
\end{array} \right.
\end{equation}
\subsection{First candidate-approximation to $S$}
Now we assume we have another model for the price of the risky asset. In this model we approximate the small jumps by a Brownian motion $B$ which is independent of $W$ and which we scale with the standard deviation of the small jumps. That is
\begin{align*}
\left\{ \begin{array}{ll}
dS_{1,\varepsilon}^{(1)}(t) &= S_{1,\varepsilon}^{(1)}(t)\Big\{a(t)dt +b(t)dW(t)+\displaystyle\int_{|z|>\varepsilon}\gamma(t,z)\widetilde{N}(dt,dz) + G(\varepsilon)\widetilde{\gamma}(t)dB(t)\Big\},\\
S_{1,\varepsilon}^{(1)}(0)&=S^{(1)}(0)= x\,.
\end{array} \right.
\end{align*}
The discounted price process is given by
\begin{align*}
d\widetilde{S}_{1,\varepsilon}(t) &= \widetilde{S}_{1,\varepsilon}(t)\Big\{(a(t)-r(t))dt +b(t)dW(t)+\int_{|z|>\varepsilon}\gamma(t,z)\widetilde{N}(dt,dz) + G(\varepsilon)\widetilde{\gamma}(t)dB(t)\Big\}.
\end{align*}
It was proven in Benth et al.\ \cite{BDK}, that the process $\widetilde{S}_{1,\varepsilon}$ converges to $\widetilde{S}$ in $L^2$ when $\varepsilon$ goes to $0$ with rate of convergence $G(\varepsilon)$. 

In the following we study the robustness of the quadratic hedging strategies towards the model choice where the price processes are 
modeled by $\widetilde{S}$ and $\widetilde{S}_{1,\varepsilon}$.
We will first show that considering the approximation $\widetilde{S}_{1,\varepsilon}$, the value of the portfolio in a quadratic hedging strategy will be written as 
a solution of a BSDEJ of type \eqref{both-approximations} with $\rho=1$. That is what explains our choice of the index $1$ in $\widetilde{S}_{1,\varepsilon}$. 
Here we choose to start with the approximation $\widetilde{S}_{1,\varepsilon}$ because it involves another Brownian motion $B$ besides the Brownian motion $W$. 
The approximation in which we replace the small jumps of the underlying price process by the Brownian motion $W$ scaled is studied in the next subsection.

The local martingale $M_{1,\varepsilon}$ in the semimartingale decomposition of $\widetilde{S}_{1,\varepsilon}$ is given by
\begin{align}
M_{1,\varepsilon}(t)&= \int_0^t b(s)\widetilde{S}_{1,\varepsilon}(s)dW(s)+\int_0^t \int_{|z|>\varepsilon}\gamma(t,z)\widetilde{S}_{1,\varepsilon}(s)\widetilde{N}(dt,dz)\nonumber\\
&\qquad \qquad + G(\varepsilon)\int_0^t \widetilde{\gamma}(s)\widetilde{S}_{1,\varepsilon}(s)dB(s)
\end{align}
and the finite variation process $A_{1,\varepsilon}$ is given by
\begin{align}\label{finite-variation}
A_{1,\varepsilon}(t)=\int_0^t\frac{a(s)-r(s)}{ \widetilde{S}_{1,\varepsilon}(s)\big(b^2(s) + \int_{|z|\geq \varepsilon}\gamma^2(s,z)\ell(dz)\big)}d\langle M_{1,\varepsilon}\rangle(s).
\end{align}
We define the process $\alpha_{1,\varepsilon}$ by 
\begin{align}\label{alpha-process-approximation1}
\alpha_{1,\varepsilon}(t) :=\frac{a(t)-r(t)}{\widetilde{S}_{1,\varepsilon}(t) \big(b^2(t) +G^2(\varepsilon)\widetilde{\gamma}^2(t)+  \int_{|z|>\varepsilon}\gamma^2(t,z)\ell(dz)\big)}, \qquad 0\leq t\leq T.
\end{align}
Thus the mean-variance trade-off process $K_{1,\varepsilon}$ is given by 
\begin{align}
K_{1,\varepsilon}(t) &= \int_0^t \alpha_{1,\varepsilon}^2(s)d\langle M_{1,\varepsilon}\rangle(s)= \int_0^t \frac{(a(s)-r(s))^2}{b^2(s)+G^2(\varepsilon)\widetilde{\gamma}^2(s)  + \int_{|z|>\varepsilon}\gamma^2(s,z)\ell(dz)}ds\nonumber\\
&=K(t),
\end{align}
in view of the definition of $G(\varepsilon)$, equation \eqref{variance-term}.
Hence the assumption \eqref{condition-mvt} ensures the existence of the FS decomposition with respect to $\widetilde{S}_{1,\varepsilon}$ for any square integrable $\mathcal{G}_T$-measurable random variable. 

Let $\xi_\varepsilon^1$ be a square integrable contingent claim as a financial derivative with underlying ${S}^{(1)}_{1,\varepsilon}$ and maturity $T$. We denote its discounted pay-off by $\widetilde{\xi}^1_\varepsilon= \frac{\xi^1_\varepsilon}{S^{(0)}(T)}$. As we have seen before, for the minimal measure to be a probability martingale measure, we have to assume that 
$$\mathcal{E}\Big(\int_0^. \alpha_{1,\varepsilon}(s)dM_{1,\varepsilon}(s)\Big)_t>0,$$
which is equivalent to
\begin{align}
\widetilde{S}_{1,\varepsilon}(t)\alpha_{1,\varepsilon}(t)\gamma(t,z)>-1, \qquad \mbox { a.e. in }  (t,z,\omega).
\end{align}
Define ${\frac{d \widetilde{\Q}^{1,\varepsilon}}{d\mathbb{P}}|}_{\mathcal{G}_t}:= \mathcal{E}\Big(\int_0^. \alpha_{1,\varepsilon}(s)dM_{1,\varepsilon}(s)\Big)_t$ and $\widetilde{V}_{1,\varepsilon}(t):=\E_{\widetilde{\Q}^{1,\varepsilon}}[\widetilde{\xi}^1_\varepsilon|\mathcal{G}_t]$. 
Then from Proposition 4.2 in Choulli et al.\ \cite{CVV},  we have the following
FS decomposition for $\widetilde{V}_{1,\varepsilon}$ written under the world measure $\mathbb{P}$ 
\begin{align}\label{portfolio1}
\widetilde{V}_{1,\varepsilon}(t)=\E_{\widetilde{\mathbb{Q}}^{1,\varepsilon}}[\widetilde{\xi}^1_\varepsilon]+\int_0^t \chi^{FS}_{1,\varepsilon}(s) d\widetilde{S}_{1,\varepsilon}(s)+ \phi_{1,\varepsilon}^{FS}(t),
\end{align}
where $\phi^{FS}_{1,\varepsilon}$ is a $\mathbb{P}$-martingale orthogonal to $M_{1,\varepsilon}$ and $\chi^{FS}_{1,\varepsilon} \in \Theta$.
Replacing $\widetilde{S}_{1,\varepsilon}$ by its expression in \eqref{portfolio1}, we get
\begin{equation*}
\left\{ \begin{array}{ll}
d\widetilde{V}_{1,\varepsilon}(t)&=\widetilde{\pi}_{1,\varepsilon}(t)(a(t)-r(t))dt +\widetilde{\pi}_{1,\varepsilon}(t)b(t)dW(t)+\widetilde{\pi}_{1,\varepsilon}(t)G(\varepsilon)\widetilde{\gamma}(t) dB(t)\\
 &\qquad + \displaystyle\int_{|z|>\varepsilon}\widetilde{\pi}_{1,\varepsilon}(t)\gamma(t,z)\widetilde{N}(dt,dz) +d\phi^{FS}_{1,\varepsilon}(t),\\
\widetilde V_{1,\varepsilon}(T)&=\widetilde{\xi}^1_\varepsilon,
\end{array} \right.
\end{equation*}
where $\widetilde{\pi}_{1,\varepsilon}=\chi^{FS}_{1,\varepsilon} \widetilde{S}_{1,\varepsilon}$. Notice that $\phi^{FS}_{1,\varepsilon}(T)$ is a $\mathcal{G}_T^\varepsilon$-measurable square integrable $\mathbb{P}$-martingale. Thus 
applying Theorem \ref{representation-theo} with $\mathbb{H}=\mathbb{G}^\varepsilon$ and using the martingale property of $\phi^{FS}_{1,\varepsilon}(T)$ we know that there exist stochastic integrands $Y_{1,\varepsilon}^{FS}$, $Y_{2,\varepsilon}^{FS}$, and $Z_{\varepsilon}^{FS}$, such that
\begin{align*}
\phi^{FS}_{1,\varepsilon}(t)&=\E[\phi_{1,\varepsilon}^{FS}(T)]+ \int_0^t Y_{1,\varepsilon}^{FS}(s)dW(s) + \int_0^t Y_{2,\varepsilon}^{FS}(s)dB(s)\nonumber\\
&\qquad \qquad + \int_0^t\int_{|z|>\varepsilon}Z_\varepsilon^{FS}(s,z) \widetilde{N}(ds,dz).
\end{align*}
Using the martingale property of $\phi_{1,\varepsilon}^{FS}$ and equation \eqref{portfolio1}, we get $\E[\phi_{1,\varepsilon}^{FS}(T)]=\E[\phi^{FS}_{1,\varepsilon}(0)]=0$.
Therefore we deduce
\begin{align}\label{phi-approximation}
\phi^{FS}_{1,\varepsilon}(t)&= \int_0^t Y_{1,\varepsilon}^{FS}(s)dW(s) + \int_0^t Y_{2,\varepsilon}^{FS}(s)dB(s) + \int_0^t\int_{|z|>\varepsilon}Z_\varepsilon^{FS}(s,z) \widetilde{N}(ds,dz).
\end{align}
In view of the orthogonality of $\phi^{FS}_{1,\varepsilon}$ with respect to $M_{1,\varepsilon}$, we have
\begin{align}\label{orthogonality-approximation}
0=Y_{1,\varepsilon}^{FS}(t)b(t)+Y^{FS}_{2,\varepsilon}G(\varepsilon)\widetilde{\gamma}(t)+\int_{|z|\geq\varepsilon}Z_{\varepsilon}^{FS}(t,z)\gamma(t,z)\ell(dz).
\end{align}
The equation we obtain for the approximating problem is thus given by
\begin{equation}\label{portfolio-approximation} 
\left\{ \begin{array}{ll}
d\widetilde{V}_{1,\varepsilon}(t)&=\widetilde{\pi}_{1,\varepsilon}(t)(a(t)-r(t))dt +(\widetilde{\pi}_{1,\varepsilon}(t) b(t)+Y_{1,\varepsilon}^{FS}(t))dW(t)\\
&\qquad \qquad + (\widetilde{\pi}_{1,\varepsilon}(t)G(\varepsilon)\widetilde{\gamma}(t) + Y_{2,\varepsilon}^{FS}(t))dB(t)\\
&\qquad \qquad  + \displaystyle\int_{|z|> \varepsilon}\big(\widetilde{\pi}_{1,\varepsilon}(t)\gamma(t,z)+Z^{FS}_\varepsilon(t,z)\big)\widetilde{N}(dt,dz),\\
\widetilde{V}_{1,\varepsilon}(T)&=\widetilde{\xi}^1_\varepsilon.
\end{array} \right.
\end{equation}
In order to apply the robustness results studied in Section \ref{robustness}, we have to prove that $\widetilde{V}$ and $\widetilde{V}_{1,\varepsilon}$
are respectively equations of type \eqref{bsdes} and \eqref{bsdes-approximation1}. That's the purpose of the next lemma. Notice that here above $\widetilde{V}_{1,\varepsilon}$,
$\widetilde{\pi}_{1,\varepsilon}$, and $\phi^{FS}_{1,\varepsilon}$ are all $\mathcal{G}_t^\varepsilon$-measurable. However since $\mathcal{G}^\varepsilon_t \subset \mathcal{G}_t$, then $\widetilde{V}_{1,\varepsilon}$,
$\widetilde{\pi}_{1,\varepsilon}$, and $\phi^{FS}_{1,\varepsilon}$ are also $\mathcal{G}_t$-measurable.
\begin{lemma}\label{lipschitz-lemma}
Let $\kappa(t) = b^2(t) +\int_{\R_0}\gamma^2(t,z)\ell(dz)$. Assume that for all $t \in [0,T]$, 
\begin{align}\label{lipschitz-condition}
\frac{|a(t)-r(t)|}{\sqrt{\kappa(t)}}\leq C, \qquad \mathbb{P}-a.s.,
\end{align}
for a positive constant $C$.
Let $\widetilde{V}$, $\widetilde{V}_{1,\varepsilon}$ be given by \eqref{portfolio}, \eqref{portfolio-approximation}, respectively. Then $\widetilde{V}$ satisfies a BSDEJ of type \eqref{bsdes} and $\widetilde{V}_{1,\varepsilon}$ satisfies a BSDEJ of type \eqref{bsdes-approximation1}.
\end{lemma}
\begin{proof} From the expression of $\widetilde{V}$, we deduce 
\begin{equation*}
\left\{ \begin{array}{ll}
d\widetilde{V}(t)&=-f(t,\widetilde{V}(t), \widetilde{Y}(t),\widetilde{Z}(t,.))+\widetilde{Y}(t)dW(t)+\displaystyle\int_{\R_0} \widetilde{Z}(t,z)\widetilde{N}(dt,dz),\\
\widetilde{V}(T)&=\widetilde{\xi},
\end{array} \right.
\end{equation*}
where 
\begin{align}\label{y-tilde}
\widetilde{Y}(t)&= \widetilde{\pi}(t) b(t)+Y^{FS}(t), \quad \widetilde{Z}(t,z)=\widetilde{\pi}(t)\gamma(t,z)+Z^{FS}(t,z), \nonumber\\[-3mm] 
 \mbox{}\\[-3mm]
&f(t,\widetilde{V}(t), \widetilde{Y}(t),\widetilde{Z}(t,.))=-\widetilde{\pi}(t)(a(t)-r(t)). \nonumber
\end{align}
We have to show that $f$ satisfies Assumptions \ref{lipschitz-assumption}(B). We first express $\widetilde{\pi}$ in terms of $\widetilde{V}$, $\widetilde{Y}$, and $\widetilde{Z}$.
Inspired by \eqref{orthogonality}, we combine $\widetilde{Y}$ and $\widetilde{Z}$ to get 
\begin{align*}
\widetilde{Y}(t)b(t)+\int_{\R_0}\widetilde{Z}(t,z)\gamma(t,z)\ell(dz)&= \widetilde{\pi}(t)\Big(b^2(t) +\int_{\R_0}\gamma^2(t,z)\ell(dz)\Big)+Y^{FS}(t)b(t)\\
&\qquad + \int_{\R_0}Z^{FS}(t,z)\gamma(t,z)\ell(dz).
\end{align*}
From \eqref{orthogonality}, we deduce that
\begin{align}\label{pi}
\widetilde{\pi}(t) = \frac{1}{\kappa(t)}\Big(\widetilde{Y}(t)b(t)+\int_{\R_0}\widetilde{Z}(t,z)\gamma(t,z)\ell(dz)\Big).
\end{align}
Hence 
$$f(t,\widetilde{V}(t), \widetilde{Y}(t),\widetilde{Z}(t,.))= - \frac{a(t)-r(t)}{\kappa(t)}\Big(\widetilde{Y}(t)b(t)+
\int_{\R_0}\widetilde{Z}(t,z)\gamma(t,z)\ell(dz)\Big).$$
Now we have to prove that $f$ is Lipschitz. Let 
\begin{equation}\label{h-process}
h(t)=\frac{a(t)-r(t)}{\kappa(t)}, \qquad t\in [0,T].
\end{equation}
We have 
\begin{align*}
|f(t,x_1,y_1,z_1)-f(t,x_2,y_2,z_2)|&\leq |h(t)|\Big[|y_1-y_2||b(t)|+\int_{\R_0}|z_1-z_2||\gamma(t,z)|\ell(dz)\Big]\\
&\leq |h(t)|[|y_1-y_2||b(t)|\\
&\qquad +(\int_{\R_0}|z_1-z_2|^2\ell(dz))^{\frac{1}{2}} (\int_{\R_0}|\gamma(t,z)|^2\ell(dz))^{\frac{1}{2}}]\\
&\leq \sqrt{\kappa(t)}|h(t)|\Big(|y_1-y_2|+\|z_1-z_2\|\Big).
\end{align*}
Thus $f$ is Lipschitz if there exists a positive constant $C$ such that 
$$\sqrt{\kappa(t)}|h(t)|=\frac{|a(t)-r(t)|}{\sqrt{\kappa(t)}}\leq C \qquad  \forall t \in [0,T] $$
and we prove the statement for $\widetilde{V}$.

From equation \eqref{portfolio-approximation}, we have 
\begin{equation*}
\left\{ \begin{array}{ll}
d\widetilde{V}_{1,\varepsilon}(t)&=-f^1(t,\widetilde{V}_{1,\varepsilon}(t), \widetilde{Y}_\varepsilon(t),\widetilde{Z}_\varepsilon(t,.),\widetilde{\zeta}_\varepsilon(t))+\widetilde{Y}_\varepsilon(t)dW(t)+\widetilde{\zeta}_\varepsilon(t)dB(t)\\
&\qquad \qquad + \displaystyle\int_{\R_0}\widetilde{Z}_\varepsilon(t,z)\widetilde{N}(dt,dz),\\
\widetilde{V}_{1,\varepsilon}(T)&=\widetilde{\xi}^1_\varepsilon,
\end{array} \right.
\end{equation*}
where 
\begin{align}\label{y-tilde-processes1}
\widetilde{Y}_\varepsilon(t)= \widetilde{\pi}_{1,\varepsilon}(t) b(t)+Y_{1,\varepsilon}^{FS}(t),&\quad
\widetilde{\zeta}_\varepsilon(t)=\widetilde{\pi}_{1,\varepsilon}(t) G(\varepsilon)\widetilde{\gamma}(t)+Y_{2,\varepsilon}^{FS}(t),\nonumber\\
\widetilde{Z}_\varepsilon(t,z)&=(\widetilde{\pi}_{1,\varepsilon}(t)\gamma(t,z)+Z^{FS}_\varepsilon(t,z))\mathbf{1}_{|z|>\varepsilon}(z),\\
f^1(t,\widetilde{V}_{1,\varepsilon}(t), \widetilde{Y}_\varepsilon(t),\widetilde{Z}_\varepsilon(t,.),\widetilde{\zeta}_\varepsilon(t))&=-\widetilde{\pi}_{1,\varepsilon}(t)(a(t)-r(t)). \nonumber
\end{align}
With the same arguments as above and using \eqref{orthogonality-approximation} we can prove that
\begin{align}\label{pi-approximation}
\widetilde{\pi}_{1,\varepsilon}(t)=\frac{1}{\kappa(t)}\Big\{\widetilde{Y}_\varepsilon(t)b(t)+\widetilde{\zeta}_\varepsilon(t)G(\varepsilon)\widetilde{\gamma}(t)+\int_{\R_0}
\widetilde{Z}_\varepsilon(t,z)\gamma(t,z)\ell(dz)\Big\}.
\end{align}
Hence
\begin{align*}
f^1(t,\widetilde{V}(t), \widetilde{Y}_\varepsilon(t),\widetilde{Z}_\varepsilon(t,.), \widetilde{\zeta}_\varepsilon(t))&= - \frac{a(t)-r(t)}{\kappa(t)}\Big(\widetilde{Y}_\varepsilon(t)b(t)+
\widetilde{\zeta}_\varepsilon(t)G(\varepsilon)\widetilde{\gamma}(t)\\
&\qquad +\int_{\R_0}\widetilde{Z}_\varepsilon(t,z)\gamma(t,z)\ell(dz)\Big)
\end{align*}
and
\begin{align*}
|f^1(t,x,y,z,\zeta)-f^1(t,\widetilde{x},\widetilde{y},\widetilde{z},\widetilde{\zeta})|&\leq |h(t)|\Big[|y-\widetilde{y}||b(t)|+\int_{\R_0}|z-\widetilde{z}||\gamma(t,z)|\ell(dz)\\
&\qquad +G(\varepsilon)|\widetilde{\gamma}(t)||\zeta-\widetilde{\zeta}|\Big]\\
&\leq \sqrt{2\kappa(t)}|h(t)|\Big(|y-\widetilde{y}|+|\zeta-\widetilde{\zeta}|+\|z-\widetilde{z}\|\Big)
\end{align*}
and we prove the statement.
\end{proof}
Notice that the assumption \eqref{lipschitz-condition} in the preceding lemma implies that the value  $K(t)$ of the MVT process defined in \eqref{lipschitz-condition} is finite for all $t\in[0,T]$ (see \eqref{condition-mvt}).

Now we present the following main result in which we prove the robustness of the value of the portfolio.
\begin{theorem}\label{portfolio-robustness}
Assume that \eqref {lipschitz-condition} holds.
Let $\widetilde{V}$, $\widetilde{V}_{1,\varepsilon}$ be given by \eqref{portfolio}, \eqref{portfolio-approximation}, respectively. Then 
$$\E\Big[\sup_{0\leq t\leq T}|\widetilde{V}(t) - \widetilde{V}_{1,\varepsilon}(t)|^2\Big] \leq C\E[|\widetilde{\xi} -\widetilde{\xi}^1_\varepsilon|^2].$$
\end{theorem}
\begin{proof}
This is an immediate result of Proposition \ref{robustness-X}. We only have to prove the assumption \eqref{stronger-Lipcshitz-condition)} on the drivers $f$ and $f^1$.
We have for all $t\in [0,T]$
\begin{align*}
&|f(t,\widetilde{V}(t),\widetilde{Y}(t),\widetilde{Z}(t,.))-f^1(t,\widetilde{V}_{1,\varepsilon}(t),\widetilde{Y}_\varepsilon(t),\widetilde{Z}_\varepsilon(t,.),Y_{2,\varepsilon}(t))|\\
&\qquad = \Big|h(t)\Big\{(\widetilde{Y}(t)-\widetilde{Y}_\varepsilon)b(t)-\widetilde{\zeta}_\varepsilon(t)G(\varepsilon)\widetilde{\gamma}(t)\\
&\qquad \qquad +\int_{\R_0}(\widetilde{Z}(t,z)-\widetilde{Z}_\varepsilon(t,z))\gamma(t,z)\ell(dz)\Big\}\Big|\\
&\qquad \leq 2|h(t)|\sqrt{\kappa(t)}\Big\{|\widetilde{Y}(t)-\widetilde{Y}_\varepsilon(t)|+\|\widetilde{Z}(t,.)-\widetilde{Z}_{\varepsilon}(t,.)\|+|\widetilde{\zeta}_\varepsilon(t)|\Big\},
\end{align*}
which proves the statement.  
\end{proof}

\begin{remark}
We used the expectation $\E[|\widetilde{\xi}-\widetilde{\xi}^\rho_\varepsilon|^2]$ to dominate the convergence results.
In finance the discounted contingent claim $\widetilde{\xi}= \frac{\xi}{S^{(0)}(T)}$ is given by the pay-off function $\xi = f(S^{(1)}(T))$. Thus we have
$$
\E[|\widetilde{\xi}-\widetilde{\xi}^\rho_\varepsilon|^2]=\E \Big[\Big|\frac{f(S^{(1)}(T))}{S^{(0)}(T)}-\frac{f(S^{(1)}_{\rho,\varepsilon}(T))}{S^{(0)}(T) } \Big|^2 \Big],  \qquad \rho=0,1,
$$
where the case $\rho=0$ refers to the second candidate-approximation of section \ref{secondapprox}.
The convergence of the latter quantity when $\varepsilon$ goes to $0$ was studied in Benth et al.\ \cite{BDK} using Fourier transform techniques. It was also studied in Kohatsu-Higa and Tankov \cite{KT} in which the authors show that adding a small variance Brownian motion to the big jumps gives better convergence results
than when we only truncate the small jumps. For this purpose the authors consider a discretization of the price models.
\end{remark}
The next theorem contains the robustness result for the amount of wealth to invest in the stock in a locally risk-minimizing strategy. 
\begin{theorem}\label{robustness-pi}
Assume that \eqref {lipschitz-condition} holds and that for all $t\in [0,T]$,
\begin{align}\label{kappa}
\inf_{t\leq s \leq T} \kappa(s)\geq K, \qquad \mathbb{P}-a.s.,
\end{align}
where $K$ is a strictly positive constant.
Let $\widetilde{\pi}$, $\widetilde{\pi}_{1,\varepsilon}$ be given by \eqref{pi}, \eqref{pi-approximation}, respectively. Then for all $t\in [0,T]$,
$$\E\Big[\int_t^T|\widetilde{\pi}(s)-\widetilde{\pi}_{1,\varepsilon}(s)|^2ds\Big] \leq C\E[|\widetilde{\xi} -\widetilde{\xi}^1_\varepsilon|^2],$$
where $C$ is a positive constant.
\end{theorem}
\begin{proof}
Using \eqref{pi} and \eqref{pi-approximation}, we have
\begin{align*}
|\widetilde{\pi}(s)-\widetilde{\pi}_{1,\varepsilon}(s)|^2&=\frac{1}{\kappa^2(s)}\Big\{(\widetilde{Y}(s)-\widetilde{Y}_\varepsilon(s))b(s)-\widetilde{\zeta}_\varepsilon(s)G(\varepsilon)\widetilde{\gamma}(s)\\
&\qquad +\int_{\R_0}(\widetilde{Z}(s,z)-\widetilde{Z}_\varepsilon(s,z))\gamma(s,z)\ell(dz)\Big\}^2\\
& \leq \frac{C}{\kappa(s)}\Big\{|\widetilde{Y}(s)-\widetilde{Y}_\varepsilon(s)|^2 +|\widetilde{\zeta}_\varepsilon(s)|^2+\int_{\R_0}|\widetilde{Z}(s,z)-\widetilde{Z}_\varepsilon(s,z)|^2\ell(dz)\Big\},
\end{align*}
where $C$ is a positive constant. Hence from Theorem \ref{bsdej-robustness} and Theorem \ref{lipschitz-lemma}, we deduce 
\begin{align*}
\E\Big[\int_t^T|\widetilde{\pi}(s)-\widetilde{\pi}_{1,\varepsilon}(s)|^2ds\Big]
& \leq \frac{C}{\inf_{t \leq s \leq T}\kappa(s)}\Big\{\E\Big[ \int_t^T |\widetilde{Y}(s)-\widetilde{Y}_\varepsilon(s)|^2ds\Big]\\
&\qquad +\E\Big[ \int_t^T|\widetilde{\zeta}_\varepsilon(s)|^2ds\Big] \\
&\qquad +\E\Big[ \int_t^T\int_{\R_0}|\widetilde{Z}(s,z)-\widetilde{Z}_\varepsilon(s,z)|^2\ell(dz)ds\Big]\Big\}\\
&\leq \widetilde{C}\E[|\widetilde{\xi}-\widetilde{\xi}^1_\varepsilon|^2]
\end{align*}
and we prove the statement.
\end{proof}
The robustness of the process $\phi^{FS}$ defined in \eqref{phi} is shown in the next theorem.
\begin{theorem}\label{robustness-phi}
Assume that \eqref {lipschitz-condition}  and \eqref{kappa} hold and for all $t\in [0,T]$,
\begin{equation}\label{gamma-kappa}
\sup_{t\leq s\leq T}\widetilde{\gamma}^2(s) \leq \widetilde{K}, \qquad \sup_{t\leq s\leq T}\kappa(s) \leq \widehat{K}<\infty, \qquad \mathbb{P}-a.s.
\end{equation}
Let $\phi^{FS}$, $\phi_{1,\varepsilon}^{FS}$ be given by \eqref{phi}, \eqref{phi-approximation}, respectively. Then for all $t \in [0,T]$, we have
$$\E\Big[|\phi^{FS}(t)-\phi_{1,\varepsilon}^{FS}(t)|^2\Big] \leq C\E[|\widetilde{\xi} -\widetilde{\xi}^1_\varepsilon|^2]+C'G^2(\varepsilon),$$
where $C$ and $C'$ are positive constants.
\end{theorem}
\begin{proof}
From \eqref{y-tilde}, \eqref{y-tilde-processes1}, Proposition \ref{robustness-X}, and Theorem \ref{robustness-pi}, we have
\begin{align*}
\E\Big[\int_t^T|Y^{FS}(s)-Y^{FS}_{1,\varepsilon}(s)|^2ds\Big]& \leq C\Big\{\E\Big[\int_t^T|\widetilde{Y}(s)-\widetilde{Y}_{1,\varepsilon}(s)|^2ds\Big]\\
&\qquad +\sup_{t \leq s \leq T}\kappa(s) \E\Big[\int_t^T|\widetilde{\pi}(s)-\widetilde{\pi}_{1,\varepsilon}(s)|^2ds\Big]\Big\}\\
&\leq \widetilde{C}\E[|\widetilde{\xi}-\widetilde{\xi}^1_\varepsilon|^2].
\end{align*}
Moreover, starting again from \eqref{y-tilde-processes1} we arrive at 
\begin{align*}
\E\Big[\int_t^T|Y^{FS}_{2,\varepsilon}(s)|^2ds\Big]&\leq C\Big\{\E[\int_t^T|\widetilde{\zeta}_\varepsilon(s)|^2ds + \sup_{t\leq s\leq T}\kappa(s) \E\Big[\int_t^T|\widetilde{\pi}(s)-\widetilde{\pi}_{1,\varepsilon}(s)|^2ds\Big]\\
&+G^2(\varepsilon)\sup_{t\leq s\leq T}\widetilde{\gamma}^2(s)\E\Big[\int_t^T|\widetilde{\pi}(s)|^2ds\Big].
\end{align*}
However from \eqref{pi} and Lemma \ref{solution-boundedness}, we get
\begin{align*}
\E\Big[\int_t^T|\widetilde{\pi}(s)|^2ds\Big]&\leq \frac{1}{\inf_{t\leq s \leq T}\kappa(s)}\Big\{\E\Big[\int_t^T\widetilde{Y}^2(s)ds\Big]+\E\Big[\int_t^T\int_{\R_0}\widetilde{Z}^2(s,z)\ell(dz)ds\Big]\Big\}\\
&\qquad \leq C\E[\widetilde{\xi}^2].
\end{align*}
Thus from Theorem \ref{bsdej-robustness} and Theorem \ref{robustness-pi} we conclude in view of assumption \eqref{gamma-kappa} 
\begin{align*}
\E\Big[\int_t^T|Y^{FS}_{2,\varepsilon}(s)|^2ds\Big]&\leq C\E[|\widetilde{\xi} -\widetilde{\xi}^1_\varepsilon|^2]+C'G^2(\varepsilon).
\end{align*}
Let $G^2(\infty)= \int_{\R_0}g^2(z)\ell(dz)$. From \eqref{y-tilde}, \eqref{y-tilde-processes1}, Theorem \ref{robustness-pi}, Lemma \ref{solution-boundedness} and Proposition \ref{robustness-X}, we obtain
\begin{align*}
&\E\Big[\int_t^T\int_{\R_0}|Z^{FS}(s,z)-Z^{FS}_\varepsilon(s,z)|^2\ell(dz)ds\Big]\\
&\qquad  \leq C \E\Big[\int_t^T\int_{\R_0}|\widetilde{Z}(s,z)-\widetilde{Z}_{\varepsilon}(s,z)|^2\ell(dz)ds\Big]\\
&\qquad \qquad + G^2(\infty)\sup_{t \leq s\leq T}\widetilde{\gamma}^2(s)\E\Big[\int_t^T|\widetilde{\pi}(s)-\widetilde{\pi}_{1,\varepsilon}(s)|^2ds\Big]\\
&\qquad \qquad +G^2(\varepsilon)\sup_{t\leq s \leq T}\widetilde{\gamma}^2(s) \E\Big[\int_t^T|\widetilde{\pi}(s)|^2ds\Big]\\
&\qquad \leq C\E[|\widetilde{\xi}-\widetilde{\xi}^1_\varepsilon|^2] + C'G^2(\varepsilon)\E[\widetilde{\xi}^2].
\end{align*}
Finally from \eqref{phi} and \eqref{phi-approximation}, we infer
\begin{align*}
\E[|\phi^{FS}(t)-\phi^{FS}_{1,\varepsilon}(t)|^2]&\leq \E\Big[\int_t^T|Y^{FS}(s)-Y_{1,\varepsilon}^{FS}(s)|^2ds\Big]+\E\Big[\int_t^T|Y_{2,\varepsilon}^{FS}(s)|^2ds\Big]\\
&\qquad +\E\Big[\int_t^T\int_{\R_0}|Z^{FS}(s,z)-Z_{\varepsilon}^{FS}(s,z)|^2\ell(dz)ds\Big]
\end{align*}
and the result follows.
\end{proof}
Let $C(t)=\phi^{FS}(t)+\widetilde{V}(0)$ and $C_{1,\varepsilon}(t)=\phi_{1,\varepsilon}^{FS}(t)+\widetilde{V}_{1,\varepsilon}(0)$.
Then the processes $C$ and $C_{1,\varepsilon}$ are the cost processes in a locally risk-minimizing strategy for $\widetilde{\xi}$ and $\widetilde{\xi}_\varepsilon^1$.
In the next corollary we prove the robustness of this cost process.
\begin{corollary}
Assume that \eqref {lipschitz-condition}, \eqref{kappa}, and \eqref{gamma-kappa} hold. Then for all $t \in [0,T]$ we have 
\begin{align*}
\E[|C(t)-C_{1,\varepsilon}(t)|^2] &\leq \widetilde{K}\E[|\widetilde{\xi} -\widetilde{\xi}^1_{\varepsilon}|^2]+ K'G^2(\varepsilon),
\end{align*}
where $\widetilde{K}$ and $K'$ are two positive constants.
\end{corollary}
\begin{proof}
From Theorem \ref{portfolio-robustness}, we deduce
$$\E\Big[|\widetilde{V}_{1,\varepsilon}(0)-\widetilde{V}(0)|^2]\leq  C\E[|\widetilde{\xi} -\widetilde{\xi}^1_\varepsilon|^2].$$
Applying the latter together with Theorem \ref{robustness-phi} we get
\begin{align*}
\E[|C(t)-C_{1,\varepsilon}(t)|^2] &= \E\Big[|(\widetilde{V}_{1,\varepsilon}(0)+\phi_{1,\varepsilon}^{FS}(t))-(\widetilde{V}(0)+\phi^{FS}(t))|^2]\\
&\leq 2 \big(\E\Big[|\widetilde{V}_{1,\varepsilon}(0)-\widetilde{V}(0)|^2]
+\E[|\phi_{1,\varepsilon}^{FS}(t)-\phi^{FS}(t)|^2]\big)\\
&\leq \widetilde{C}\E[|\widetilde{\xi} -\widetilde{\xi}^1_\varepsilon|^2]+ C'G^2(\varepsilon).
\end{align*}
\end{proof}
In the next section we present a second-candidate approximation to $S$ and we study the robustness of the quadratic hedging strategies.
\subsection{Second-candidate approximation to $S$}\label{secondapprox}
In the second candidate-approximation to $S$, we truncate the small jumps of the jump-diffusion and we replace them by the Brownian motion $W$ which we scale with the 
standard deviation of the small jumps. We obtain the following dynamics for the approximation  
\begin{align*}
\left\{ \begin{array}{ll}
dS_{0,\varepsilon}^{(1)}(t) &= S_{0,\varepsilon}^{(1)}(t)\Big\{a(t)dt +(b(t)+G(\varepsilon)\widetilde{\gamma}(t))dW(t)+\displaystyle\int_{|z|>\varepsilon}\gamma(t,z)\widetilde{N}(dt,dz)\Big\},\\
S_{0,\varepsilon}^{(1)}(0)&= S^{(1)}(0)=x.
\end{array}\right.
\end{align*}
The discounted price process is given by
\begin{align*}
d\widetilde{S}_{0,\varepsilon}(t) &= \widetilde{S}_{0,\varepsilon}(t)\Big\{(a(t)-r(t))dt +(b(t)+G(\varepsilon)\widetilde{\gamma}(t))dW(t)+\int_{|z|>\varepsilon}\gamma(t,z)\widetilde{N}(dt,dz)\Big\}.
\end{align*}
It is easy to show that $\widetilde{S}_{0,\varepsilon}(t)$ converges to $\widetilde{S}(t)$ in $L^2$ when $\varepsilon$ goes to $0$ with rate of convergence $G(\varepsilon)$.
\begin{remark}
Notice that in this paper we consider two types of approximations that are truly different.
In the candidate-approximation $S_{1,\varepsilon}$, the variance of the continuous part is given by $b^2(t)+G^2(\varepsilon)\widetilde{\gamma}^2(t)$, 
which is the same as the sum of the variance of the small jumps and the variance of the continuous part in $S$. We study this approximation by embedding the original model  solution into a larger filtration $\mathbb{G}$.
In the candidate-approximation $S_{0,\varepsilon}$\,, the variance of the continuous part is given by $(b(t)+G(\varepsilon)\widetilde{\gamma}(t))^2$,
which is different from the sum of the variance of the small jumps and the variance of the continuous part in $S$. 
In this approximation the solution is considered in the original filtration $\mathbb{F}$.  
Thus it is accordingly to the preferences of the trader to decide which candidate-approximation to choose. If one insists on working under the filtration $\mathbb{F}$, then 
one could also select a third candidate-approximation as a variation of $S_{0,\varepsilon}$\,. In fact one could also choose a function $\widetilde{G}(\varepsilon)$ in $S_{0,\varepsilon}$ in such a way that the approximation has the same variance as in the original process. 
In other words we can choose $\widetilde{G}(\varepsilon)$ such that
$$(b(t)+\widetilde{G}(\varepsilon)\widetilde{\gamma}(t))^2=b^2(t)+G^2(\varepsilon)\widetilde{\gamma}^2(t).$$
The latter is equivalent to choosing $\widetilde{G}^2(\varepsilon)= G(\varepsilon)(2b+G(\varepsilon)\widetilde{\gamma})$, which is clearly vanishing when $\varepsilon$ goes to $0$. Notice that from Theorem \ref{bsdej-robustness} the fact that $\widetilde{G}^2(\varepsilon)$ is vanishing when $\varepsilon$ goes to $0$ is enough to prove the robustness of the quadratic hedging strategies. 

We remark that our study has pointed out the role of the filtration in the study of approximations of the small jumps by a Brownian motion appropriately scaled. Indeed this aspect of the study was not discussed before.
\end{remark}
The local martingale $M_{0,\varepsilon}$ in the semimartingale decomposition of $\widetilde{S}_{0,\varepsilon}$ is given by
\begin{align*}
M_{0,\varepsilon}(t)&= \int_0^t (b(s)+G(\varepsilon)\widetilde{\gamma}(t))\widetilde{S}_{0,\varepsilon}(s)dW(s)+\int_0^t \int_{|z|>\varepsilon}\gamma(t,z)\widetilde{S}_{0,\varepsilon}(s)\widetilde{N}(dt,dz)\,.
\end{align*}
We define the process $\alpha_{0,\varepsilon}$ by 
\begin{align*}
\alpha_{0,\varepsilon}(t) :=\frac{a(t)-r(t)}{\widetilde{S}_{0,\varepsilon}(t) \Big\{(b(t) +G(\varepsilon)\widetilde{\gamma}(t))^2+ \int_{|z|>\varepsilon}\gamma^2(t,z)\ell(dz)\Big\}}\,, \qquad 0\leq t\leq T.
\end{align*}
Thus the mean-variance trade-off process $K_{0,\varepsilon}$ is given by 
\begin{align*}
K_{0,\varepsilon}(t) &= \int_0^t \alpha_{0,\varepsilon}^2(s)d\langle M_{0,\varepsilon}\rangle(s)= \int_0^t \frac{(a(s)-r(s))^2}{(b(s)+G(\varepsilon)\widetilde{\gamma}(s))^2  + \int_{|z|>\varepsilon}\gamma^2(s,z)\ell(dz)}ds\,.
\end{align*}
Hence we have to assume that $K_{0,\varepsilon}(t)$ is bounded uniformly by a positive constant to ensure the existence of the FS decomposition with respect to $\widetilde{S}_{0,\varepsilon}$ for any square integrable $\mathcal{F}_T$-measurable random variable. 

Let $\xi_\varepsilon^0$ be a square integrable contingent claim as a financial derivative with underlying $\widetilde{S}_{0,\varepsilon}$. We denote the discounted pay-off of $\xi^0_\varepsilon$ by $\widetilde{\xi}^0_\varepsilon= \frac{\xi^0_\varepsilon}{S^{(0)}(T)}$.
As we have seen before, for the signed minimal martingale measure to exist as a probability martingale measure, we have to assume that 
\begin{align*}
\widetilde{S}_{0,\varepsilon}(t)\alpha_{0,\varepsilon}(t)\gamma(t,z)>-1, \qquad \mbox { a.e. in }  (t,z,\omega).
\end{align*}
Define $\frac{d\widetilde{\Q}^{0,\varepsilon}}{d\mathbb{P}}_{|\mathcal{F}_t}:= \mathcal{E}\Big(\int_0^. \alpha_{0,\varepsilon}(s)dM_{0,\varepsilon}(s)\Big)_t$ and $\widetilde{V}_{0,\varepsilon}(t):=\E_{\widetilde{\Q}^{0,\varepsilon}}[\widetilde{\xi}^0_\varepsilon|\mathcal{F}_t]$. Following the same steps as before, we get the following equation for the value of the portfolio 
\begin{equation*}
\left\{ \begin{array}{ll}
d\widetilde{V}_{0,\varepsilon}(t)&=\widetilde{\pi}_{0,\varepsilon}(t)(a(t)-r(t))dt +\widetilde{\pi}_{0,\varepsilon}(t)(b(t)+G(\varepsilon)\widetilde{\gamma}(t))dW(t)\\
 &\qquad +\displaystyle \int_{|z|>\varepsilon}\widetilde{\pi}_{0,\varepsilon}(t)\gamma(t,z)\widetilde{N}(dt,dz) +d\phi^{FS}_{0,\varepsilon}(t),\\
\widetilde V_{0,\varepsilon}(T)&=\widetilde{\xi}^0_\varepsilon,
\end{array} \right.
\end{equation*}
where $\widetilde{\pi}_{0,\varepsilon}=\chi^{FS}_{0,\varepsilon} \widetilde{S}_{0,\varepsilon}$ and $\chi_{0,\varepsilon}^{FS} \in \Theta$. Since 
$\phi^{FS}_{0,\varepsilon}(T)$ is a $\mathcal{F}_T^\varepsilon$-measurable square integrable $\mathbb{P}$-martingale, then 
applying Theorem \ref{representation-theo} with $\mathbb{H}=\mathbb{F}^\varepsilon$ and using the martingale property of $\phi^{FS}_{0,\varepsilon}(T)$ we know that there exist stochastic integrands $Y_{\varepsilon}^{FS}$ and $Z_{\varepsilon}^{FS}$, such that
\begin{align}\label{phi-approximation1}
\phi^{FS}_{0,\varepsilon}(t)&=\E[\phi_{0,\varepsilon}^{FS}(T)]+ \int_0^t Y_{\varepsilon}^{FS}(s)dW(s) + \int_0^t\int_{|z|>\varepsilon}Z_\varepsilon^{FS}(s,z) \widetilde{N}(ds,dz).
\end{align}
Using the same arguments as for $\phi^{FS}_{1,\varepsilon}$ we can prove that $\E[\phi_{0,\varepsilon}^{FS}(T)]=\E[\phi_{0,\varepsilon}^{FS}(0)]=0$. 
In view of the orthogonality of $\phi^{FS}_{0,\varepsilon}$ with respect to $M_{0,\varepsilon}$, we have
\begin{align}\label{orthogonality-approximation1}
0=Y_{\varepsilon}^{FS}(t)[b(t)+G(\varepsilon)\widetilde{\gamma}(t)] +\int_{|z|\geq\varepsilon}Z_{\varepsilon}^{FS}(t,z)\gamma(t,z)\ell(dz).
\end{align}
The equation we obtain for the approximating problem is thus given by
\begin{equation}\label{portfolio-approximation1}
\left\{ \begin{array}{ll}
d\widetilde{V}_{0,\varepsilon}(t)&=\widetilde{\pi}_{0,\varepsilon}(t)(a(t)-r(t))dt +\big(\widetilde{\pi}_{0,\varepsilon}(t)[b(t)+G(\varepsilon)\widetilde{\gamma}(t)]+Y_{\varepsilon}^{FS}(t)\big)dW(t)\\
&\qquad \qquad  + \displaystyle\int_{|z|> \varepsilon}\big(\widetilde{\pi}_{0,\varepsilon}(t)\gamma(t,z)+Z^{FS}_\varepsilon(t,z)\big)\widetilde{N}(dt,dz),\\
\widetilde{V}_{0,\varepsilon}(T)&=\widetilde{\xi}^0_\varepsilon.
\end{array} \right.
\end{equation}
In the next lemma we prove that $\widetilde{V}_{0,\varepsilon}$ satisfies the set of equations of type \eqref{bsdes-approximation}.
\begin{lemma}\label{lipschitz-lemma1}
Assume that \eqref{lipschitz-condition} holds.
Let $\widetilde{V}_{0,\varepsilon}$ be given by \eqref{portfolio-approximation1}. Then 
$\widetilde{V}_{0,\varepsilon}$ satisfies a BSDEJ of type \eqref{bsdes-approximation}.
\end{lemma}
\begin{proof} 
We rewrite equation \eqref{portfolio-approximation1} as
\begin{equation*}
\left\{ \begin{array}{ll}
d\widetilde{V}_{0,\varepsilon}(t)&=-f^0(t,\widetilde{V}_{0,\varepsilon}(t), \widetilde{Y}_\varepsilon(t),\widetilde{Z}_\varepsilon(t,.))+\widetilde{Y}_{\varepsilon}(t)dW(t) 
+\displaystyle \int_{\R_0}\widetilde{Z}_\varepsilon(t,z)\widetilde{N}(dt,dz),\\
\widetilde{V}_{0,\varepsilon}(T)&=\widetilde{\xi}^0_\varepsilon,
\end{array} \right.
\end{equation*}
where we introduce the processes $\widetilde{Y}_{\varepsilon}$, $\widetilde{Z}_\varepsilon$ and the function $f^0$ by
\begin{align}\label{y-tilde-processes2}
\widetilde{Y}_{\varepsilon}(t)&= \widetilde{\pi}_{0,\varepsilon}(t)[b(t)+G(\varepsilon)\widetilde{\gamma}(t)]+Y_{\varepsilon}^{FS}(t),\nonumber\\
\widetilde{Z}_\varepsilon(t,z)&=(\widetilde{\pi}_{0,\varepsilon}(t)\gamma(t,z)+Z^{FS}_\varepsilon(t,z))\mathbf{1}_{|z|>\varepsilon}(z),\\
f^0(t,\widetilde{V}_{0,\varepsilon}(t), \widetilde{Y}_\varepsilon(t),\widetilde{Z}_\varepsilon(t,.))&=-\widetilde{\pi}_{0,\varepsilon}(t)(a(t)-r(t)).\nonumber
\end{align}
With the same arguments as above and using \eqref{orthogonality-approximation1} we can prove that
\begin{align}\label{pi-approximation1}
\widetilde{\pi}_{0,\varepsilon}(t)=\frac{1}{\kappa(t)}\Big\{\widetilde{Y}_\varepsilon(t)b(t)+\widetilde{\zeta}_\varepsilon(t)G(\varepsilon)\widetilde{\gamma}(t)+\int_{\R_0}
\widetilde{Z}_\varepsilon(t,z)\gamma(t,z)\ell(dz)\Big\}.
\end{align}
Hence
\begin{align*}
f^0(t,\widetilde{V}(t), \widetilde{Y}_{\varepsilon}(t),\widetilde{Z}_\varepsilon(t,.))&= - \frac{a(t)-r(t)}{\kappa(t)}\Big(\widetilde{Y}_{\varepsilon}(t)[b(t)+G(\varepsilon)\widetilde{\gamma}(t)]+\int_{\R_0}\widetilde{Z}_\varepsilon(t,z)\gamma(t,z)\ell(dz)\Big)
\end{align*}
and it is easy to show that $f^0$ is Lipschitz when \eqref{lipschitz-condition} holds. This proves the statement.
\end{proof}
Now we present the following theorem in which we prove the robustness of the value of the portfolio.
\begin{theorem}
Assume that \eqref {lipschitz-condition} holds.
Let $\widetilde{V}$, $\widetilde{V}_{0,\varepsilon}$ be given by \eqref{portfolio}, \eqref{portfolio-approximation1}, respectively. Then we have
$$\E\Big[\sup_{0\leq t\leq T}|\widetilde{V}(t) - \widetilde{V}_{0,\varepsilon}(t)|^2\Big] \leq C\E[|\widetilde{\xi} -\widetilde{\xi}^0_\varepsilon|^2].$$
\end{theorem}
\begin{proof}
Following the same steps as in the proof of Theorem \ref{portfolio-robustness}, we can show that $f^0$ satisfies condition \eqref{condition-f-zero} and we prove the statement by applying Proposition \ref{robustness-X}. 
\end{proof}
In the next theorem we prove the robustness of the amount of wealth to invest in a locally risk-minimizing strategy.
\begin{theorem}\label{robustness-pi1}
Assume that \eqref {lipschitz-condition} holds and that for all $t\in [0,T]$
\begin{align}\label{kappa1}
\inf_{t\leq s \leq T} \kappa(s) \geq K, \qquad \sup_{t\leq s\leq T}\widetilde{\gamma}^2(s) \leq \widetilde{K}, \qquad \mathbb{P}-a.s.,
\end{align}
where $K$ and $\widetilde{K}$ are strictly positive constants.
Let $\widetilde{\pi}$, $\widetilde{\pi}_{0,\varepsilon}$ be given by \eqref{pi}, \eqref{pi-approximation1}, respectively. Then 
$$\E\Big[\int_t^T|\widetilde{\pi}(s)-\widetilde{\pi}_{0,\varepsilon}(s)|^2ds\Big] \leq C\E[|\widetilde{\xi} -\widetilde{\xi}^0_\varepsilon|^2] +C'G^2(\varepsilon),$$
where $C$ and $C'$ are positive constants.
\end{theorem}
\begin{proof}
We have 
\begin{align*}
|\widetilde{\pi}(s)-\widetilde{\pi}_{1,\varepsilon}(s)|^2&=\frac{1}{\kappa^2(s)}\Big\{(\widetilde{Y}(s)-\widetilde{Y}_{\varepsilon}(s))(b(s)+G(\varepsilon)\widetilde{\gamma}(s))-\widetilde{Y}(s)G(\varepsilon)\widetilde{\gamma}(s)\\
&\qquad +\int_{\R_0}(\widetilde{Z}(s,z)-\widetilde{Z}_\varepsilon(s,z))\gamma(s,z)\ell(dz)\Big\}^2\\
& \leq \frac{C}{\kappa(s)}\Big\{|\widetilde{Y}(s)-\widetilde{Y}_{\varepsilon}(s)|^2 +G^2(\varepsilon)\widetilde{\gamma}^2(s)|\widetilde{Y}(s)|^2\\
&\qquad +\int_{\R_0}|\widetilde{Z}(s,z)-\widetilde{Z}_\varepsilon(s,z)|^2\ell(dz)\Big\},
\end{align*}
where $C$ is a positive constant. Hence from Theorem \ref{bsdej-robustness}, we deduce 
\begin{align*}
\E\Big[\int_t^T|\widetilde{\pi}(s)-\widetilde{\pi}_{1,\varepsilon}(s)|^2ds\Big]
& \leq \frac{C}{\inf_{t \leq s \leq T}\kappa(s)}\Big\{\E\Big[ \int_t^T |\widetilde{Y}(s)-\widetilde{Y}_\varepsilon(s)|^2ds\Big]\\
&\qquad +G^2(\varepsilon)\sup_{t\leq s\leq T}\widetilde{\gamma}^2(s)\E\Big[ \int_t^T|\widetilde{Y}(s)|^2ds\Big] \\
&\qquad +\E\Big[ \int_t^T\int_{\R_0}|\widetilde{Z}(s,z)-\widetilde{Z}_\varepsilon(s,z)|^2\ell(dz)ds\Big]\Big\}\\
&\leq \widetilde{C}\E[|\widetilde{\xi}-\widetilde{\xi}^0_\varepsilon|^2]+C'G^2(\varepsilon)\E[|\widetilde{\xi}|^2]
\end{align*}
and we prove the statement.
\end{proof}
In the next theorem we deal with the robustness of the process $\phi^{FS}$.
\begin{theorem}
Assume that \eqref {lipschitz-condition}  and \eqref{kappa1} hold.
Let $\phi^{FS}$, $\phi_{0,\varepsilon}^{FS}$ be given by \eqref{phi}, \eqref{phi-approximation1}, respectively. Then for all $t \in[0,T]$ we have 
$$\E\Big[|\phi^{FS}(t)-\phi_{0,\varepsilon}^{FS}(t)|^2\Big] \leq C\E[|\widetilde{\xi} -\widetilde{\xi}^0_\varepsilon|^2]+C'G^2(\varepsilon),$$
where $C$ and $C'$ are positive constants.
\end{theorem}
\begin{proof}
From \eqref{y-tilde}, \eqref{y-tilde-processes2}, and Proposition \ref{robustness-X}, we have
\begin{align*}
\E\Big[\int_t^T|Y^{FS}(s)-Y^{FS}_\varepsilon(s)|^2ds\Big]& \leq C\Big\{\E\Big[\int_t^T|\widetilde{Y}(s)-\widetilde{Y}_\varepsilon(s)|^2ds\Big]\\
&\qquad +\sup_{t \leq s \leq T}\kappa(s) \E\Big[\int_t^T|\widetilde{\pi}(s)-\widetilde{\pi}_{0,\varepsilon}(s)|^2ds\Big]\Big\}\\
&\leq \widetilde{C}\E[|\widetilde{\xi}-\widetilde{\xi}^0_\varepsilon|^2].
\end{align*}
Combining \eqref{y-tilde}, \eqref{y-tilde-processes2}, Theorem \ref{robustness-pi1}, Lemma \ref{solution-boundedness}, and Proposition \ref{robustness-X}, we arrive at
\begin{align*}
&\E\Big[\int_t^T\int_{\R_0}|Z^{FS}(s,z)-Z^{FS}_\varepsilon(s,z)|^2\ell(dz)ds\Big]\\
&\qquad  \leq C \E\Big[\int_t^T\int_{\R_0}|\widetilde{Z}(s,z)-\widetilde{Z}_{\varepsilon}(s,z)|^2\ell(dz)ds\Big]\\
&\qquad \qquad + G^2(\infty)\sup_{t \leq s\leq T}\widetilde{\gamma}^2(s)\E\Big[\int_t^T|\widetilde{\pi}(s)-\widetilde{\pi}_{0,\varepsilon}(s)|^2ds\Big]\\
&\qquad \qquad +G^2(\varepsilon)\sup_{t\leq s \leq T}\widetilde{\gamma}^2(s) \E\Big[\int_t^T|\widetilde{\pi}(s)|^2ds\Big]\\
&\qquad \leq C\E[|\widetilde{\xi}-\widetilde{\xi}^0_\varepsilon|^2] + C'G^2(\varepsilon)\E[\widetilde{\xi}^2].
\end{align*}
\end{proof}
Define the cost process in the risk-minimizing strategy for $\widetilde{\xi}_\varepsilon^0$ by 
$$C_{0,\varepsilon}(t)=\phi_{0,\varepsilon}^{FS}(t)+\widetilde{V}_{0,\varepsilon}(0).$$
Then an obvious implication of the last theorem is the robustness of the cost process and it is easy to show that under the same conditions of the 
last theorem we have for all $t \in [0,T]$,
 \begin{align*}
\E[|C(t)-C_{0,\varepsilon}(t)|^2] &\leq \widetilde{K}\E[|\widetilde{\xi} -\widetilde{\xi}^0_{\varepsilon}|^2]+ K'G^2(\varepsilon),
\end{align*}
where $\widetilde{K}$ and $K'$ are two positive constants.

\subsection{A note on the robustness of the mean-variance hedging strategies}
A mean-variance hedging strategy is a self-financing strategy $\widetilde{V}$ for which we do not impose the replication requirement $\widetilde{V}(T)=\widetilde{\xi}$.
However we insist on the self-financing constraint. In this case we define the shortfall or loss from hedging $\widetilde{\xi}$ by
\begin{equation*}
 \widetilde{\xi}-\widetilde{V}(T)=\widetilde{\xi}-\widetilde{V}(0)-\int_0^T\widetilde{\Gamma}(s)d\widetilde{S}(s), \quad \widetilde{V}(0) \in \R, \quad \widetilde{\Gamma} \in \Theta\,.
 \end{equation*}
In order to obtain the MVH strategy one has to minimize the latter quantity in the $L^2$-norm by choosing $(\widetilde{V}(0),\widetilde{\Gamma}) \in (\R, \Theta)$.
Schweizer \cite{SC1} gives a formula for the number of risky assets in a MVH strategy where he assumes that the so-called extended mean-variance tradeoff process is deterministic.

In this paper, given the dynamics of the stock price process $S$, the process $A$ defined in \eqref{finite-variation1} is continuous. Thus the mean-variance tradeoff process and the extended mean-variance tradeoff process defined in Schweizer \cite{SC1} coincide.
Therefore applying Theorem 3 and Corollary 10 in Schweizer \cite{SC1} 
and assuming that the mean-variance tradeoff process $K$ is deterministic, the discounted number of risky assets in a mean-variance hedging strategy is given by 
\begin{equation}\label{mvh-strategy}
\widetilde{\Gamma}(t)= \widetilde{\chi}^{FS}(t)+\alpha(t)\Big(\widetilde{V}(t)-\widetilde{V}(0)-\int_0^t\widetilde{\Gamma}(s)d\widetilde{S}(s)\Big),
\end{equation}
where $\alpha$ and $\widetilde{\chi}^{FS}$ are as defined in \eqref{alpha-process} and \eqref{portfolio-general}.
Moreover the minimal martingale measure and the mean-variance martingale measure coincide (see Schweizer \cite{SC}) and in this case $\widetilde{V}(t)=\E_{\widetilde{Q}}[\widetilde{\xi}|\mathcal{F}_t]$, $0\leq t\leq T$, where $\widetilde{Q}$ is the minimal martingale measure. Multiplying \eqref{mvh-strategy} by $\widetilde{S}$
we obtain the following equation for the amount of wealth in a mean-variance hedging strategy
$$\widetilde{\Upsilon}(t)= \widetilde{\pi}(t)+h(t)\Big(\widetilde{V}(t)-\widetilde{V}(0)-\int_0^t\frac{\widetilde{\Upsilon}(s)}{\widetilde{S}(s)}d\widetilde{S}(s)\Big),$$
where $h$ is given by \eqref{h-process}. Since $K$ is deterministic then $a$, $b$, $r$, $\gamma$, and thus $h$ should be deterministic. 
We consider the approximating stock process $\widetilde{S}_{1,\varepsilon}$. The amount of wealth in a mean-variance hedging strategy associated to $\widetilde{S}_{1,\varepsilon}$ is given by
$$\widetilde{\Upsilon}_{1,\varepsilon}(t)= \widetilde{\pi}_{1,\varepsilon}(t)+h(t)\Big(\widetilde{V}_{1,\varepsilon}(t)-\widetilde{V}_{1,\varepsilon}(0)-\int_0^t\frac{\widetilde{\Upsilon}_{1,\varepsilon}(s)}{\widetilde{S}_{1,\varepsilon}}d\widetilde{S}_{1,\varepsilon}(s)\Big).$$
Before we show the robustness of the mean-variance hedging strategies. We present the following lemma in which we show the boundedness in $L^2$
of $\widetilde{\Upsilon}$.
\begin{lemma} \label{boundedness-upsilon}
Assume that the mean-variance tradeoff process $K$ \eqref{mvt} is deterministic and that \eqref{lipschitz-condition} holds true. Then 
for all $t\in [0,T]$,
$$
\E[\widetilde{\Upsilon}^2(t)]\leq C(T)\E[\xi^2],
$$
where $C(T)$ is a positive constant depending on $T$. 
\end{lemma}
\begin{proof}
Applying It\^{o} isometry and H\"{o}lder inequality, we get
\begin{align*}
\E[\widetilde{\Upsilon}^2(t)]&\leq \E[\widetilde{\pi}^2(t)]+C'h^2(t)\Big(\E[\widetilde{V}^2(t)]+\E[\widetilde{V}^2(0)]\\
&\quad+ \int_0^t\E[\widetilde{\Upsilon}^2(s)]\{(a(s)-r(s))^2+b^2(s)+\int_{\R_0}\gamma^2(s,z)\ell(dz)\}ds \Big),
\end{align*}
where $C'$ is a positive constant. Using Lemma \ref{solution-boundedness}, Lemma \ref{lipschitz-lemma}, and equation \eqref{pi}, the result follows applying Gronwall's inequality.
\end{proof}
In the following theorem we prove the robustness of the amount of wealth in a mean-variance hedging strategy.
\begin{theorem} Assume the mean-variance tradeoff process is deterministic and that \eqref{lipschitz-condition} and \eqref{kappa} hold. Then for all $t\in [0,T]$,
$$\E\Big[|\widetilde{\Upsilon}(t)-\widetilde{\Upsilon}_{1,\varepsilon}(t)|^2]\leq C\E[|\widetilde{\xi}-\widetilde{\xi}^1_\varepsilon|^2]+\widetilde{C}G^2(\varepsilon) .$$
\end{theorem}
\begin{proof}
We have 
\begin{align*}
&|\widetilde{\Upsilon}(t)-\widetilde{\Upsilon}_{1,\varepsilon}(t)|\\
&\leq|\widetilde{\pi}(t)-\widetilde{\pi}_{1,\varepsilon}(t)|
 +|h(t)|\Big(|V(t)-V_{1,\varepsilon}(t)|+|V(0)-V_{1,\varepsilon}(0)|\\
&\qquad +\int_0^t |\widetilde{\Upsilon}(s) -\widetilde{\Upsilon}_{1,\varepsilon}(s)||a(s)-r(s)|ds+|\int_0^t (\widetilde{\Upsilon}(s) -\widetilde{\Upsilon}_{1,\varepsilon}(s))b(s)dW(s)|\\
&\qquad +|\int_0^t\int_{|z|> \varepsilon}(\widetilde{\Upsilon}(s)-\widetilde{\Upsilon}_{1,\varepsilon}(s))\gamma(s,z)\widetilde{N}(ds,dz)|\\
&\qquad+G(\varepsilon)|\int_0^t (\widetilde{\Upsilon}_{1,\varepsilon}(s)-\widetilde{\Upsilon}(s)\widetilde{\gamma}(s)dB(s)|\\
&\qquad+|\int_0^t\int_{|z|\leq \varepsilon}\widetilde{\Upsilon}(s)\gamma(s,z)\widetilde{N}(ds,dz)|
 +G(\varepsilon)|\int_0^t \widetilde{\Upsilon}(s)\widetilde{\gamma}(s)dB(s)|\Big).
\end{align*}
Using It\^{o} isometry and H\"{o}lder inequality, we get
\begin{align*}
&\E[|\widetilde{\Upsilon}(t)-\widetilde{\Upsilon}_{1,\varepsilon}(t)|^2]\\
&\quad \leq \E[|\widetilde{\pi}(t)-\widetilde{\pi}_{1,\varepsilon}(t)|^2]+\widetilde{C}h^2(t)\Big(\E[|V(t)-V_{1,\varepsilon}(t)|^2]+\E[|V(0)-V_{1,\varepsilon}(0)|^2]\\
&\quad +\int_0^t \E[|\widetilde{\Upsilon}(s) -\widetilde{\Upsilon}_{1,\varepsilon}(s)|^2]\Big(|a(s)-r(s)|^2+|b(s)|^2+\int_{\R_0}|\gamma(s,z)|^2\ell(dz)\Big)ds\\
&\qquad +G^2(\varepsilon) \int_0^t \E[\widetilde{\Upsilon}^2(s)]\widetilde{\gamma}^2(s) ds \Big),
\end{align*}
where $\widetilde{C}$ is a positive constant.
Using Theorem \ref{portfolio-robustness}, Theorem \ref{robustness-pi}, and Lemma \ref{boundedness-upsilon}
the result follows applying Gronwall's inequality.
\end{proof}

We proved in this section that when the mean-variance trade-off process $K$ defined in \eqref{mvt} is deterministic, then the value of the portfolio and the amount of wealth in a mean-variance hedging strategy are robust towards the choice of the model. The same robustness result holds true when we consider the stock price process $\widetilde{S}_{0,\varepsilon}$.
We do not present this result since it follows the same lines as the approximation $\widetilde{S}_{1,\varepsilon}$.
\section{Conclusion}
In this paper we consider different models for the price process. Then using BSDEJs we proved that the locally risk-minimizing and the mean-variance hedging strategies are 
robust towards the choice of the model. Our results are given in terms of estimates containing $\E[|\widetilde{\xi}-\widetilde{\xi}^\rho_\varepsilon|^2]$, which is a quantity well studied
by Benth et al.\ \cite{BDK} and Kohatsu-Higa and Tankov \cite{KT}

We have specifically studied two types of approximations of the price $S$ and suggested a third approximation. 
 It is also possible to consider other approximations to the price $S$. For example we can truncate the small jumps without adding a Brownian motion. In that case, based on the 
 robustness of the BSDEJs, we can also prove the robustness of quadratic hedging strategies. 
 Another approximation is to add to the L\'{e}vy process a scaled Brownian motion. This type of approximation was discussed  and justified in a paper by Benth et al. \cite{BDK}. 
 
As far as further investigations are concerned, we consider in another paper a time-discretization of these different price models and study the convergence of the quadratic hedging strategies related to each of these time-discretized price models to the quadratic hedging strategies related to the original continuous time model. Moreover, we are concerned with the characterization of the approximating models which give the best convergence rates when the robustness of quadratic hedging strategies is taken into account.

\section{Appendix: existence and uniqueness of BSDEJs}\label{appendix}
\noindent {\bf{Proof of Theorem \ref{existence-1}.}}
First, we prove the result in $\widetilde{H}^2_{T,\beta} \times H^2_{T,\beta}\times \widehat{H}^2_{T,\beta}$
with the norm 
$$\|(X_\varepsilon,Y_\varepsilon,Z_\varepsilon)\|^2_{\widetilde{H}^2_{T,\beta} \times H^2_{T,\beta}\times \widehat{H}^2_{T,\beta}}=\|X_\varepsilon\|^2_{\widetilde{H}^2_{T,\beta}}+\|Y_\varepsilon\|^2_{H^2_{T,\beta}}+\|Z_\varepsilon\|^2_{\widehat{H}^2_{T,\beta}}.$$
The proof is based on a fixed point theorem.
Let $(U_\varepsilon, V_\varepsilon, K_{\varepsilon}) \in \widetilde{H}_{T,\beta}^2 \times H_{T,\beta} \times \widehat{H}_{T,\beta}^2$ and define $X_\varepsilon$ as follows
\begin{equation}\label{conditional-expectation}
X_\varepsilon(t)= \E\Big[\xi^0_\varepsilon + \int_t^Tf^0(s,U_\varepsilon(s),V_\varepsilon(s), K_{\varepsilon}(s,\cdot))ds |\mathcal{F}_t\Big], \qquad 0\leq t\leq T.
\end{equation}
Applying Theorem \ref{representation-theo} with $\mathbb{H}=\mathbb{F}$, to the square integrable $\mathcal{F}_T$-measurable random variable $$\xi^0_\varepsilon + \int_0^Tf^0(s,U_\varepsilon(s),V_\varepsilon(s), K_{\varepsilon}(s,\cdot))ds,$$ we know that there exist 
$Y_\varepsilon \in H^2_{T,\beta}$ and $Z_\varepsilon \in \widehat{H}^2_{T,\beta}$ such that 
\begin{align}\label{martingale-representation-bsdej}
&\xi^0_\varepsilon + \int_0^Tf^0(s,U_\varepsilon(s),V_\varepsilon(s), K_{\varepsilon}(s,\cdot))ds\nonumber\\
&\qquad = \E\Big[\xi^0_\varepsilon + \int_0^Tf^0(s,U_\varepsilon(s),V_\varepsilon(s), K_{\varepsilon}(s,\cdot))ds\Big] + \int_0^TY_\varepsilon(s)dW(s)\nonumber\\
&\qquad \qquad  +  \int_0^T\int_{\R_0}Z_\varepsilon(s,z)\widetilde{N}(ds,dz).
\end{align}
Taking the conditional expectation with respect to $\mathcal{F} _t$ and using the martingale property, we get
\begin{align*}
&\E\Big[\xi^0_\varepsilon + \int_0^Tf^0(s,U_\varepsilon(s),V_\varepsilon(s), K_{\varepsilon}(s,\cdot))ds|\mathcal{F}_t\Big]\\
&\qquad =\E\Big[\xi^0_\varepsilon + \int_0^Tf^0(s,U_\varepsilon(s),V_\varepsilon(s), K_{\varepsilon}(s,\cdot))ds\Big]   + \int_0^tY_\varepsilon(s)dW(s)\\
&\qquad \qquad + \int_0^t\int_{\R_0}Z_{\varepsilon}(s,z)\widetilde{N}(ds,dz).
\end{align*}
Since the integral over the interval $(0,t)$ is $\mathcal{F}_t$-measurable, we find
\begin{align*}
&\E\Big[\xi^0_\varepsilon + \int_0^Tf^0(s,U_\varepsilon(s),V_\varepsilon(s), K_{\varepsilon}(s,\cdot))ds|\mathcal{F}_t\Big]\\
&\qquad = \int_0^tf^0(s,U_\varepsilon(s),V_\varepsilon(s), K_{\varepsilon}(s,\cdot))ds \\
&\qquad \qquad + \E\Big[\xi^0_\varepsilon + \int_t^Tf^0(s,U_\varepsilon(s),V_\varepsilon(s), K_{\varepsilon}(s,\cdot))ds|\mathcal{F}_t\Big].\\
\end{align*}
Thus by the definition of $X_\varepsilon$, equation \eqref{conditional-expectation}, we have
\begin{align*}
X_\varepsilon(t)&=X_\varepsilon(0)-\int_0^tf^0(s,U_\varepsilon(s),V_\varepsilon(s), K_{\varepsilon}(s,\cdot))ds + \int_0^tY_\varepsilon(s)dW(s)\\
&\qquad +\int_0^t\int_{\R_0}Z_{\varepsilon}(s,z)\widetilde{N}(ds,dz),
\end{align*}
from which by combining with \eqref{martingale-representation-bsdej}, we deduce that 
\begin{align*}
X_\varepsilon(t)&=\xi^0_\varepsilon + \int_t^Tf^0(s,U_\varepsilon(s),V_\varepsilon(s), K_{\varepsilon}(s,\cdot))ds -\int_t^T Y_\varepsilon(s)dW(s) \\
&\qquad - \int_t^T\int_{\R_0}Z_{\varepsilon}(s,z)\widetilde{N}(ds,dz).
\end{align*}
This relation defines a mapping $\phi: \widetilde{H}^2_{T,\beta} \times H^2_{T,\beta}\times \widehat{H}^2_{T,\beta}  \longrightarrow 
\widetilde{H}^2_{T,\beta} \times H^2_{T,\beta}\times \widehat{H}^2_{T,\beta}$ with
$(X_\varepsilon,Y_\varepsilon,Z_\varepsilon)=\phi(U_\varepsilon,V_\varepsilon,K_{\varepsilon})$.
We may conclude that $(X_\varepsilon,Y_\varepsilon,Z_\varepsilon) \in \widetilde{H}^2_{T,\beta} \times H^2_{T,\beta}\times \widehat{H}^2_{T,\beta}$ solves the BSDEJ \eqref{bsdes-approximation} if and only if it is a fixed point of $\phi$. 

Hereto we prove that $\phi$ is a strict contraction on $\widetilde{H}^2_{T,\beta} \times H^2_{T,\beta}\times \widehat{H}^2_{T,\beta}$ 
for a suitable $\beta >0$.
Let $(U_\varepsilon,V_\varepsilon,K_{\varepsilon})$  and $(\widehat{U}_\varepsilon,\widehat{V}_\varepsilon,\widehat{K}_\varepsilon)$ be two elements of $\widetilde{H}^2_{T,\beta} \times H^2_{T,\beta}\times \widehat{H}^2_{T,\beta}$ and set $\phi(U_\varepsilon,V_\varepsilon,K_{\varepsilon})=(X_\varepsilon,Y_\varepsilon,Z_\varepsilon)$ and 
$\phi(\widehat{U}_\varepsilon,\widehat{V}_\varepsilon,\widehat{K}_\varepsilon)=(\widehat{X}_\varepsilon,\widehat{Y}_\varepsilon,\widehat{Z}_\varepsilon)$. Denote 
$(\bar{U}_\varepsilon,\bar{V}_\varepsilon,\bar{K}_\varepsilon)=(U_\varepsilon-\widehat{U}_\varepsilon, V_\varepsilon-\widehat{V}_\varepsilon,
K_{\varepsilon}-\widehat{K}_\varepsilon)$ and $(\bar{X}_\varepsilon,\bar{Y}_\varepsilon,\bar{Z}_\varepsilon)=(X_\varepsilon-\widehat{X}_\varepsilon, Y_\varepsilon-\widehat{Y}_\varepsilon,Z_\varepsilon-\widehat{Z}_\varepsilon)$.
Applying the It\^{o} formula to $\e^{\beta s} \bar{X}_\varepsilon(s)$, it follows that
\begin{align*}
\e^{\beta t}\bar{X}^2_\varepsilon(t)&=- \int_t^T\beta\e^{\beta s}\bar{X}^2_\varepsilon(s)ds
 + 2\int_t^T\bar{X}_\varepsilon(s)\e^{\beta s}\{f^0(s,U_\varepsilon(s),V_\varepsilon(s),K_{\varepsilon}(s,\cdot))\\
&\qquad  -f^0(s,\widehat{U}_\varepsilon(s),\widehat{V}_\varepsilon(s), \widehat{K}_\varepsilon(s,\cdot))\}ds - \int_t^T2e^{\beta s}\bar{X}_\varepsilon(s)\bar{Y}_\varepsilon(s)dW(s)\\
&\qquad -\int_t^Te^{\beta s}\bar{Y}^2(s)ds- \int_t^T\int_{\R_0}\e^{\beta s} \bar{Z}_{\varepsilon}^2(s,z)\ell(dz)ds\\
&\qquad  - \int_t^T\int_{\R_0}\Big\{ \e^{\beta s} \bar{Z}^2_{\varepsilon}(s,z)-2\e^{\beta s} \bar{Z}_{\varepsilon}(s,z)\bar{X}_\varepsilon(s)\Big\} \widetilde{N}(ds,dz).
\end{align*}
Taking the expectation, we get
\begin{align*}
&\E\Big[\e^{\beta t}\bar{X}^2_\varepsilon(t)\Big]+\E\Big[\int_t^Te^{\beta s}(\bar{Y}_\varepsilon(s))^2ds\Big]+\E\Big[\int_t^T\int_{\R_0}\e^{\beta s} \bar{Z}^2_{\varepsilon}(s,z)\ell(dz)ds\Big]\\ 
&\qquad = -\beta\E\Big[\int_t^T\e^{\beta s}\bar{X}_\varepsilon^2(s)ds\Big]\\
&\qquad \qquad + 2\E\Big[\int_t^T\e^{\beta s}\bar{X}_\varepsilon(s)\{f^0(s,U_\varepsilon(s),V_\varepsilon(s),K_{\varepsilon}(s,\cdot))-f^0(s,\widehat{U}_\varepsilon(s),\widehat{V}_\varepsilon(s),\widehat{K}_\varepsilon(s,\cdot))\}ds\Big].
\end{align*}
Since by Assumptions \eqref{lipschitz-assumption}(B), $f$ is Lipschitz we can dominate the right hand side above as follows
\begin{align*}
&\E\Big[\e^{\beta t}\bar{X}^2_\varepsilon(t)\Big]+\E\Big[\int_t^Te^{\beta s}\bar{Y}^2_\varepsilon(s)ds\Big]+\E\Big[\int_t^T\int_{\R_0}\e^{\beta s} \bar{Z}^2_{\varepsilon}(s,z)\ell(dz)ds\Big]\\ 
&\qquad \leq -\beta\E\Big[\int_t^T\e^{\beta s}\bar{X}^2_\varepsilon(s)ds\Big]\\
&\qquad \qquad + 2C\E\Big[\int_t^T\e^{\beta s}\bar{X}_\varepsilon(s)\Big\{|\bar{U}_\varepsilon(s)|+|\bar{V}_\varepsilon(s)|+\big(\int_{\R_0}\bar{K}_\varepsilon^2(s,z)\ell(dz)\big)^{\frac{1}{2}}\Big\}ds\Big].
\end{align*}
Using the fact that for every $k>0$ and $a,b \in\R$ we have that $2ab\leq ka^2+\frac{b^2}{k}$ and $(a+b+c)^2\leq 3(a^2+b^2+c^2)$, we obtain
\begin{align*}
&\E\Big[\e^{\beta t}\bar{X}^2_\varepsilon(t)\Big]+\E\Big[\int_t^Te^{\beta s}\bar{Y}^2_\varepsilon(s)ds\Big]+\E\Big[\int_t^T\int_{\R_0}\e^{\beta s} \bar{Z}^2_{\varepsilon}(s,z)\ell(dz)ds\Big]\\ &\qquad \leq (6C^2-\beta)\E\Big[\int_t^T\e^{\beta s}\bar{X}_\varepsilon^2(s)ds\Big]\\
&\qquad \qquad + \frac{1}{2}\E\Big[\int_t^T\e^{\beta s}\Big\{\bar{U}^2_\varepsilon(s)+
\bar{V}^2_\varepsilon(s)+\int_{\R_0}\bar{K}^2_\varepsilon(s,z)\ell(dz)\Big\}ds\Big].
\end{align*}
Taking $\beta=6C^2+1$ and noting that $\E\Big[\e^{\beta t}|\bar{X}_\varepsilon(t)|^2\Big]\geq 0$, we obtain  
$$\|(\bar{X}_\varepsilon,\bar{Y}_\varepsilon, \bar{Z}_\varepsilon)\|_{\widetilde{H}^2_{T,\beta} \times H^2_{T,\beta}\times \widehat{H}^2_{T,\beta}}^2\leq\frac{1}{2}\|(\bar{U}_\varepsilon, \bar{V}_\varepsilon, \bar{K}_\varepsilon)\|_{\widetilde{H}^2_{T,\beta} \times H^2_{T,\beta}\times \widehat{H}^2_{T,\beta}}^2,$$
from which we proved that $\phi$ is a strict contraction on $\widetilde{H}^2_{T,\beta} \times H^2_{T,\beta}\times \widehat{H}^2_{T,\beta}$ equipped with the norm $\|.\|_{\widetilde{H}^2_{T,\beta} \times H^2_{T,\beta}\times \widehat{H}^2_{T,\beta}}$ if $\beta= 6C^2+1$. Since the $\beta$-norms are equivalent, this holds for all $\beta>0$. Thus we prove that $\phi$ has a unique fixed point. Hence there exists a unique solution in the space $\widetilde{H}^2_{T,\beta} \times H^2_{T,\beta}\times \widehat{H}^2_{T,\beta}$ to the BSDEJ \eqref{bsdes-approximation}. One can prove that $X_\varepsilon \in S^2_{T,\beta}$ using Burkholder's inequality (see Tang and Li \cite {TL} for more details) and the statement follows.




\end{document}